\documentclass[11pt,a4paper]{amsart}
\usepackage{mathrsfs}
\usepackage{syntonly}
\usepackage{amsmath}
\usepackage{amsthm}
\usepackage{amsfonts}
\usepackage{amssymb}
\usepackage{latexsym}
\usepackage{amscd,amssymb,amsopn,amsmath,amsthm,graphics,amsfonts,mathrsfs,accents,enumerate,verbatim,calc}
\usepackage[dvips]{graphicx}
\usepackage[colorlinks=true,linkcolor=red,citecolor=blue]{hyperref}
\usepackage[all]{xy}
\usepackage{enumitem}

\date{}
\hypersetup
{colorlinks=true,linkcolor=black}
\allowdisplaybreaks[4] \footskip=15pt
\renewcommand{\uppercasenonmath}[1]{}

\numberwithin{equation}{section} \theoremstyle{plain}
\newtheorem{lem}{Lemma}[section]
\newtheorem{cor}[lem]{Corollary}
\newtheorem{prop}[lem]{Proposition}
\newtheorem{thm}[lem]{Theorem}

\newtheorem{definition}[lem]{Definition}
\newtheorem{Ex}[lem]{Example}
\newtheorem{Quest}[lem]{Question}
\newtheorem{Property}[lem]{Property}
\newtheorem{Properties}[lem]{Properties}
\newtheorem{Subprops}{}[lem]
\newtheorem{Para}[lem]{}

\newtheorem{rem}[lem]{Remark}

\newenvironment{ex}{\begin{Ex}\rm}{\end{Ex}}

\newtheorem*{ack*}{ACKNOWLEDGEMENTS}

\newcommand{\oo}{\otimes}
\newcommand{\pf}{\noindent\begin {proof}}
\newcommand{\epf}{\end{proof}}

\newcommand{\A}{\mathcal{A}}

\newcommand{\C}{\mathcal{C}}
\newcommand{\U}{\mathcal{U}}

\newcommand{\mto}{\rightarrowtail}
\newcommand{\eto}{\twoheadrightarrow}

\newcommand{\s}{\stackrel}

\newcommand{\B}{\mathcal{B}}

\newcommand{\Ext}{{\rm Ext}}

\newcommand{\Hom}{{\rm Hom}}

\newcommand{\ccc}[3]{{
  \left(\begin{smallmatrix} {#1}   \\   {#2} \\\end{smallmatrix}\right)_{#3}}}

\begin{document}
\begin{center}
{\Large  \bf  Gluing and lifting exact model structures for the recollement of exact categories}

\vspace{0.5cm}  Jiangsheng Hu$^{a}$, Haiyan Zhu$^{b}$\footnote{Corresponding author.\\
\indent Jiangsheng Hu was supported by the NSF of China (12171206) and the Natural Science Foundation of Jiangsu Province (BK20211358). Haiyan Zhu was supported by Zhejiang Provincial Natural Science Foundation of China (LY18A010032) and the NSF of China
(12271481). Rongmin Zhu was supported by the NSF of China (12201223).} and Rongmin Zhu$^{c}$

$^{a}$School of Mathematics, Hangzhou Normal University,
 Hangzhou 311121, China\\

$^{b}$College of Science, Zhejiang University of Technology, Hangzhou 310023, China\\

$^{c}$School of Mathematical Sciences, Huaqiao University, Quanzhou 362021, China\\

E-mails: jiangshenghu@hotmail.com, hyzhu@zjut.edu.cn and rongminzhu@hotmail.com \\
\end{center}
\bigskip
\centerline { \bf  Abstract}
\leftskip10truemm \rightskip10truemm \noindent
In this paper, we first provide an explicit procedure to glue together hereditary exact model structures for the recollement of exact categories. To that end, we use the notion of cotorsion pairs and we investigate the gluing of complete hereditary cotorsion pairs along the recollement of exact categories. Moreover, we study liftings of recollements of hereditary exact model structures to recollements of their associated homotopy categories. This leads to a new method to produce recollements of triangulated categories. Applications are given to contraderived categories, projective stable derived categories and stable categories of Gorenstein injective modules over an upper triangular matrix ring.
\leftskip10truemm \rightskip10truemm \noindent
\\[2mm]

\leftskip0truemm \rightskip0truemm

\section{Introduction}
The notion of a cotorsion pair goes back to \cite{SL}, which has been defined originally in the category of abelian groups, and then in an abelian category or an exact category. It got an enormous impulse thanks to the discovery by Hovey \cite{ho02} of the one-to-one correspondence between abelian model structures and certain cotorsion pairs in abelian categories. Later on, Gillespie demonstrated in \cite{gj11} that the above Hovey's one-to-one correspondence naturally carries over to a correspondence between hereditary exact model structures and cotorsion pairs in a weakly idempotent complete exact category. For short, we will call this a \emph{WIC exact category} in this paper.
Therefore, the theory of exact model structures concerns the case of when $\C$ is a WIC exact category and there is a model structure on $\C$ that is compatible with the exact structure. The upshot of working with a hereditary exact model structure is that its homotopy category is canonically triangulated, in fact, it coincides with the stable category of a Frobenius category (see \cite{gj11,gj16,hv99} for examples).

Recollements were first introduced in the setting of triangulated categories by Beilinson, Bernsteinand Deligne \cite{BBD} and then generalized to the level of abelian categories (see for instance \cite{FP2004,GKP,Kuhn,Psaroudakis}). Recently, Wang, Wei and Zhang \cite{Wand-wei-Zhang} give a generalization of recollements of abelian categories, which they called recollements of exact categories. Roughly speaking, a recollement is a short exact sequence of triangulated or exact categories where the functors involving admit both left and right adjoints. Such a recollement situation of exact categories is denoted throughout the paper by the
following diagram
$$\xymatrix@C=50pt@R=50pt{\ \mathcal{A}\ar[r]^{i_*}&\ \mathcal{C}\ar@/^1pc/[l]^{i^!}\ar@/_1pc/[l]_{i^*}\ar[r]^{j^*} &\ \mathcal{B}\ar@/^1pc/[l]^{j_*}\ar@/_1pc/[l]_{j_!} }\eqno{(1.1)}$$
of WIC exact categories and additive functors satisfying the compatibility conditions in \cite[Definition 3.1]{Wand-wei-Zhang}.
In this case one says that $(\mathcal{A},\mathcal{C},\mathcal{B})$ is a \emph{recollement of exact categories}.

It should be noted that recollements of exact categories (in particular abelian categories) appear quite naturally in various settings and are omnipresent in representation theory (see \cite{GKP,Psaroudakis}). For instance, any idempotent elemente in a ring $R$ induces a recollement situation between the module categories over the rings $R$, $R/ReR$ and $eRe$ (see \cite[Example 2.7]{Psaroudakis}). Moreover, it has been shown in Theorem \ref{thm:3.5} and Example \ref{ex:3.1} that a recollement of abelian categories under some mild conditions can induce a nontrivial recollement of exact categories (it is no longer a recollement of abelian categories).

However, the existence of recollements of triangulated categories often is difficult to establish and then plays an important role in geometry of singular spaces \cite{BBD}, representation theory \cite{Angeleri,Cline,konig} and homological conjectures \cite{Chenxi0,Chenxi,Happel2,Z}.
In the recent ten years, there are a lot of interesting work on the constructions of recollements
of triangulated categories such as derived categories of ordinary rings or differential graded rings, stable categories of Frobenius categories, or more generally, homotopy category of exact model structures (see for instance \cite{Bazzoni2019,Chenxi0,Chenxi,Chenxi-JPAA,GC,gj16,gj16.1}). Motivated by the discussion so far, we study the following general questions.

\textbf{Question 1.1.}  Let $(\mathcal{A},\mathcal{C},\mathcal{B})$ be a recollement of exact categories.
 \begin{enumerate}
\item[(1)] How can we glue together
hereditary exact model structures in $\A$ and $\B$ to obtain a hereditary exact model structure in $\C$?

\item[(2)] When these model structures can be lifted to a recollement of their associated triangulated homotopy categories?
 \end{enumerate}

In order to answer these questions,  we first need to take up the following question because of the one-to-one correspondence between exact model structures and certain complete cotorsion pairs in a WIC exact category (see \cite[Corollary 3.4]{gj11}).

\textbf{Question 1.2.}  Giving a recollement $(\mathcal{A},\mathcal{C},\mathcal{B})$ of exact categories, how can we glue together complete hereditary
cotorsion pairs in $\A$ and $\B$ to obtain a complete hereditary cotorsion pair in $\C$?

Recall that for a right exact functor $T:\mathcal{B}\rightarrow\mathcal{A}$ between abelian categories, there exists an abelian category, denoted by $(T\downarrow\mathcal{A})$, consisting of all triples $\left(\begin{smallmatrix}X\\Y\end{smallmatrix}\right)_{\varphi}$ where $\varphi: T(Y)\rightarrow X$ is a morphism in $\mathcal{A}$. We note that this new abelian category is called a \emph{comma category} in \cite{Marmaridis} and forming a comma category along a given functor is a standard way to glue two categories.  We refer to \cite{Huzhu,Psaroudakis} for a detailed discussion on this matter. Recently, Hu and Zhu characterized when complete hereditary cotorsion pairs in abelian categories $\mathcal{A}$ and $\mathcal{B}$ can induce complete hereditary cotorsion pairs in $(T\downarrow\A)$ (see \cite[Proposition 3.4]{Huzhu}). If $(\mathcal{A},\mathcal{C},\mathcal{B})$ is a recollement of abelian categories and $i^{!}$ is an exact functor in Question 1.2, then the abelian category $\mathcal{C}$ is equivalent to the comma category $(i^{!}j_{!}\downarrow\A)$ (see \cite[Proposition 8.9]{FP2004} or \cite[Proposition 3.1]{GKP}). So \cite[Proposition 3.4]{Huzhu} gives an answer to Question 1.2 provided that $(\mathcal{A},\mathcal{C},\mathcal{B})$ is a recollement of abelian categories such that $i^{!}$ is an exact functor.

The main aim of this paper is to provide more answers to the above questions. To state our results precisely, we first introduce some notation and definitions.

Let $T:\mathcal{D}_1\rightarrow \mathcal{D}_2$ be a functor between WIC exact categories, and let $\mathcal{Y}$ be a subcategory of $\mathcal{D}_1$. The functor $T$ is called \emph{$\mathcal {Y}$-exact} if $T$ preserves the exactness of the admissible exact sequence $B\mto B'\eto Y$ in $\mathcal{D}_1$ with $Y\in\mathcal {Y}$. Here we denote admissible monomorphisms by $\rightarrowtail$ and denote admissible epimorphisms by $\twoheadrightarrow$.

Let $\mathcal{X}$ be a subcategory of $\mathcal{A}$ and $\mathcal{Y}$ a subcategory of $\mathcal{B}$ in the recollement (1.1). We set
\begin{displaymath}
\mathscr{N}^{\mathcal{X}}_{\mathcal{Y}}:=\{C\in{\C} \ | \ i^{!}(C)\in{\mathcal{X}}, \ j^{*}(C)\in{\mathcal{Y}}\},
\end{displaymath}
\begin{displaymath}
\mathscr{M}^{\mathcal{X}}_{\mathcal{Y}}:=\{C\in{\C} \ | \ i^{*}(C)\in{\mathcal{X}}, \ j^{*}(C)\in{\mathcal{Y}}, \ \varepsilon_{C}: j_{!} j^{*}(C)  \mto C \ \textrm{is} \ \textrm{an} \ \textrm{admissible} \ \textrm{monomorphism}\},
\end{displaymath}
where $\varepsilon:j_{!} j^{*}\rightarrow 1_{\mathcal{C}}$ is the counit of the adjoint pair $(j_{!},j^{*})$ in the recollement (1.1).

The next result conveys that one can glue together complete hereditary cotorsion pairs from $\mathcal{A}$ and $\mathcal{B}$ to $\mathcal{C}$ in the recollement (1.1), which provides a partial answer to Question 1.2 properly.

\begin{thm}\label{thm:1.1} Let $(\mathcal{A},\mathcal{C},\mathcal{B})$ be a recollement of exact categories with $i^{!}$ an exact functor.
 Assume that $(\mathcal{U}',\mathcal{V}')$ and $(\mathcal{U}'',\mathcal{V}'')$ are complete hereditary cotorsion pairs in $\mathcal{A}$ and $\mathcal{B}$, respectively.
 Set $\mathcal{U}=\mathscr{M}^{\mathcal{U}'}_{\mathcal{U}''}$ and $\mathcal{V}=\mathscr{N}_{\mathcal{V}''}^{\mathcal{V}'}$. If $\mathcal{C}$ has enough projective and injective objects and $j_{!}$ is $\mathcal{U}''$-exact, then $(\mathcal{U},\mathcal{V})$ is a complete hereditary cotorsion pair in $\C$.
\end{thm}

A few comments on Theorem \ref{thm:1.1} are in order. First, it generalizes Lemma 3.3 and Proposition 3.4 in \cite{Huzhu}. More precisely, in \cite{Huzhu}, one of key arguments in the proof is that all objects in the comma category $(T\downarrow\mathcal{A})$ can be represented clearly by
the objects in $\mathcal{A}$ and $\mathcal{B}$, while in our general context we do not have this fact and therefore must avoid this kind
of arguments. So, the idea of proving Theorem \ref{thm:1.1} will be different from the one in \cite{Huzhu} (see Remark \ref{remark:3.9}).

Second, it should be pointed out that the condition ``$(\mathcal{A},\mathcal{C},\mathcal{B})$ is a recollement of exact categories with $i^{!}$ an exact functor" is natural and very often met. More specifically, one can construct the desired recollement with $i^{!}$ an exact functor from any recollement of abelian categories (see Theorem \ref{thm:3.5} and Example \ref{ex:3.1}). During the course of the proof of Theorem \ref{thm:3.5}, for any recollement (1.1) of exact categories, we will show that $i^{!}$ is an exact functor if and only if $i^{*}j_{*}j^{*}=0$, which refines a result obtained by Franjou and Pirashvili in \cite{FP2004} (see Proposition
\ref{prop:2.5} and Remark \ref{remark:2.6}).

Last, we employ an example in \cite{ZH} to illustrate Theorem \ref{thm:1.1} as follows: (1) The exactness of the functor $i^!$ cannot be omitted in general; (2) The condition ``$\mathcal{C}$ has enough projective and injective objects and $j_{!}$ is $\mathcal{U}''$-exact" really occurs (see Example \ref{ex:3.2}).

Our next result, providing a partial answer to Question 1.1, can be stated as follows. Its proof is based on Theorem \ref{thm:1.1} and a deep result of Gillespie  about constructing a hereditary exact model structure from two cotorsion pairs in a WIC exact category (see \cite[Theorem 1.1]{gj14}).

\begin{thm}\label{thm:main1} Let $(\mathcal{A},\mathcal{C},\mathcal{B})$ be a recollement of exact categories  with $i^{!}$ an exact functor, and let $\mathcal{M}_{\mathcal{A}}=(\mathcal{U}'_{1},\mathcal{W}',\mathcal{V}'_{2})$ and $\mathcal{M}_{\mathcal{B}}=(\mathcal{U}''_{1},\mathcal{W}'',\mathcal{V}''_{2})$ be hereditary exact model structures on $\mathcal{A}$ and $\mathcal{B}$, respectively.  Set $\mathcal{U}_1=\mathscr{M}^{\mathcal{U}'_{1}}_{\mathcal{U}''_{1}}$, $\mathcal{V}_1=\mathscr{N}_{\mathcal{W}''\cap\mathcal{V}''_{2}}^{\mathcal{W}'\cap\mathcal{V}'_{2}}$, $\mathcal{U}_2=\mathscr{M}^{\mathcal{U}'_{1}\cap\mathcal{W}'}_{\mathcal{U}''_{1}\cap\mathcal{W}''}$ and $\mathcal{V}_2=\mathscr{N}_{\mathcal{V}''_{2}}^{\mathcal{V}'_{2}}$. Assume that $\mathcal{C}$ has enough projective and injective objects, $j_{!}$ is $\mathcal{U}''_{1}$-exact and  $\mathcal{U}_{1}\cap\mathcal{V}_{1}=\mathcal{U}_{2}\cap\mathcal{V}_{2}$.

\begin{enumerate}
\item There is a hereditary exact model structure $\mathcal{M}_{\mathcal{C}}=(\mathcal{U}_{1},\mathcal{W},\mathcal{V}_{2})$ on $\mathcal{C}$, where the class $\mathcal{W}$ is given by
\begin{align*}
 \mathcal{W }&=\{ X \in\mathcal{C} \mid \exists~\text{an admissible exact sequence}~ X\mto R\eto Q~with~R\in{\mathcal{V}_1},~Q\in{\mathcal{U}_2} \}\\
&=\{ X \in\mathcal{C} \mid \exists~\text{an admissible exact sequence}~ R'\mto Q'\eto X~with~R'\in{\mathcal{V}_1},~Q'\in{\mathcal{U}_2}\}.
\end{align*}

\item We have the following recollement of triangulated categories

$$\xymatrixcolsep{5pc}\xymatrix{\mathrm{Ho}(\mathcal{M}_{\mathcal{A}})
  \ar[r]^{L(i_{\ast})\cong R(i_{\ast})}&
  \ar@/_2pc/[l]_{L(i^{\ast})}\ar@/^2pc/[l]^{R(i^{!})}\mathrm{Ho}(\mathcal{M}_{\mathcal{C}})
\ar[r]^{L(j^{\ast})\cong R(j^{\ast})}&\ar@/_2pc/[l]_{L(j_{!})}\ar@/^2pc/[l]^{R(j_{\ast})}
\mathrm{Ho}(\mathcal{M}_{\mathcal{B}}),}$$
where $L(i^{\ast})$, $L(i_{\ast})$, $L(j_{!})$, $L(j^{\ast})$, $R(i_{\ast})$,  $R(j^{\ast})$, $R(i^{!})$ and $R(j_{\ast})$ are the total derived functors of those in {\rm (1.1)}.
\end{enumerate}
\end{thm}

Recently, Gao, Koenig and Psaroudakis \cite{GKP} showed that
ladders of certain height of recollements of abelian categories
allow to construct recollements of triangulated categories; Georgios and Psaroudakis \cite{GC} used Quillen model structures to show a systematic method to lift recollements of projective (resp., injective) hereditary abelian model structures (see \cite[Setup 6.5]{GC}) to recollements of their associated homotopy categories. It should be noted that our work offers a different perspective. More precisely, we first glue together complete hereditary cotorsion pairs for the recollement of exact categories, then provide an explicit procedure to glue together hereditary exact model structures (not necessarily projective or injective abelian model structures) for this recollement. As a result, we can build recollements of triangulated categories by lifting recollement of hereditary exact model categories to recollements of their associated homotopy categories.

As an application of Theorem \ref{thm:main1}, for any upper triangular matrix ring, we obtain recollements of contraderived categories, projective stable derived categories and stable categories of Gorenstein injective modules. (see Corollaries \ref{corollary:5.5}, \ref{corollary:5.6} and \ref{corollary:5.8}).

The structure of this paper is organized as follows. In Section \ref{pre}, we give some terminologies and some preliminary results which are needed for our proof. In Section 3,  we provide a method to construct recollements of exact categories from recollements of abelian categories. In Section 4, we prove Theorems \ref{thm:1.1} and \ref{thm:main1} mentioned in the introduction.  Finally, some applications of Theorem \ref{thm:main1} on upper triangular matrix rings are given in Section 5.

\section{Preliminaries}\label{pre}
The assumptions, the notation, and the definitions from this section will be used
throughout the paper.

\subsection{Exact categories}  The concept of an exact category is originally due to Quillen \cite{Quillen}, but the common reference for a simple
axiomatic description is \cite[Appendix A]{Keller} and an extensive treatment of the concept is also given in \cite{TBh10}. Roughly speaking, an \emph{exact category} is a pair $(\mathcal{A}, \mathcal{E})$ where $\mathcal{A}$ is an additive category and $\mathcal{E}$ is a class of ``short exact sequences": That is,
an actual kernel-cokernel pair $A\s{i}\mto B\s{p}\eto C$. In what follows, we call such a sequence an \emph{admissible exact sequence}, and call $A\mto B$ (resp., $B\eto C$) an \emph{admissible monomorphism} (resp., \emph{admissible
epimorphism}). Many authors use the alternate terms \emph{conflation}, \emph{inflation} and \emph{deflation}.
The class $\mathcal{E}$ of admissible exact sequences must satisfy exact axioms, for details, we refer the reader to \cite[Definition 2.1]{TBh10}, which are inspired by the properties of short exact sequences in any abelian category.

We often write $\mathcal{A}$ instead of $(\mathcal{A}, \mathcal{E})$ when we consider only one exact structure on $\mathcal{A}$.  An exact category $\mathcal{A}$ is called \emph{weakly idempotent complete} if every split monomorphism has a cokernel and
every split epimorphism has a kernel (see for instance \cite[Definition 2.2]{gj11}). For convenience, we will call this a \emph{WIC exact category}.

\begin{lem}\cite[Proposition 2.3]{gj11} The following are true for any WIC exact category:
\begin{enumerate}
\item If $gf$ is an admissible monomorphism, then $f$ is an admissible monomorphism.
\item If $gf$ is an admissible epimorphism, then $g$ is an admissible epimorphism.
\end{enumerate}
\end{lem}

Next, we recall the following definition, which is a particular case of Definitions 2.9 and 2.12 in \cite{Wand-wei-Zhang}.

\begin{definition}  Let $\A$ be a WIC exact category. A sequence $A\s{f}\to B\s{g}\to C$ in $\A$ is said to be \emph{right exact} if  there exist an admissible exact sequence $K\s{h_{2}}\mto B\s{g}\eto C$ and an admissible epimorphism $h_1: A\eto K$ such that $f = h_{2}h_{1}$.  Moreover, an additive covariant functor $F:\A \to \B$ between exact categories is called a \emph{right exact functor} if it takes those right exact sequences in $\A$ to sequences of the same ilk in $\B$. Dually, one can also define the left exact sequences and left exact functors.
\end{definition}

Let $n$ be a natural number. Following \cite{TBh10}, an $n+1$-term sequence $$X_{n}\s{f_{n}}\mto X_{n-1}\s{f_{n-1}}\to \cdots\to X_{2}\s{f_{2}}\to X_{1}\s{f_{1}}\eto X_{0}$$ in $\A$ is called \emph{exact} if it satisfies the following conditions:

\begin{enumerate}
\item Both $X_{n}\mto X_{n-1}\eto \ker(f_{n-2})$ and $\ker(f_{1})\mto X_{1}\eto X_{0}$ are admissible exact sequences.
 \item $\ker(f_{i})\mto X_{i}\eto \ker(f_{i-1})$ is also an admissible exact sequence for any $2\leq i\leq {n-2}$.
\end{enumerate}

Similarly, one can also define the exact sequence $\cdots\to X_{2}\s{f_{2}}\to X_{1}\s{f_{1}}\eto X_{0}$ with infinite length.
\subsection{Cotorsion pairs in exact categories} In this section, we always assume that $\A$ is an exact category. Recall from \cite{gj11} that a pair of subcategories $(\mathcal{X},\mathcal{Y})$ of $\mathcal{A}$ is said to be a \emph{cotorsion pair} if
\begin{center}
$\mathcal{X}={^{\perp}\mathcal{Y}}:=\{X\in{\mathcal{A}} \ |\ \Ext_{\mathcal{A}}^{1}(X,Y)=0 \ \textrm{for each} \ Y\in{\mathcal{Y}} \}$,
$\mathcal{Y}={\mathcal{X}^{\perp}}:=\{Y\in{\mathcal{A}} \ |\ \Ext_{\mathcal{A}}^{1}(X,Y)=0 \ \textrm{for each} \ X\in{\mathcal{X}} \}$.
\end{center}

A cotorsion pair $(\mathcal{X}, \mathcal{Y})$ is called \emph{complete} \cite{gj11} if for each $M\in{\mathcal{A}}$ there exist admissible exact sequences in $\mathcal{A}$
\begin{center}
$Y_{M}\mto X_{M}\s{f_{M}}\eto M \ \textrm{and}\  M\s{g^{M}}\mto Y^{M}\eto X^{M}$
\end{center}
such that $X_{M},X^{M}\in{\mathcal{X}}$ and $Y_{M},Y^{M}\in{\mathcal{Y}}$.
In this case, $f_{M}$ is called a \emph{special $\mathcal{X}$-precover}, while $g^{M}$ is called a \emph{special $\mathcal{Y}$-preenvelope}.

Recall that an object $P$ in $\A$ is called \emph{projective} provided that any admissible epimorphism ending at $P$ splits.
The exact category $\A$ is said to have \emph{enough projective objects} provided that each object $X$ fits into an admissible epimorphism $d: P \twoheadrightarrow X$ with $P$ projective. Dually one has the notions of injective objects and exact categories with enough injective objects.

The following lemma is essentially taken from \cite[Lemma 5.20]{Gobel}, where a variation of it
appears. The proof given there carries over to the present situation.

\begin{lem}\label{lem:salce-lemma} {\rm (Salce's Lemma)} Assume that $\mathcal{A}$ has enough projective objects and injective objects, and $(\mathcal{X}, \mathcal{Y})$ is a
cotorsion pair in $\mathcal{A}$. Then the following are equivalent.
\begin{enumerate}
\item $(\mathcal{X}, \mathcal{Y})$ is complete.

\item $\mathcal{X}$ is special precovering in $\mathcal{A}$.

\item $\mathcal{Y}$ is special preenveloping in $\mathcal{A}$.
\end{enumerate}
\end{lem}

A cotorsion pair $(\mathcal{X}, \mathcal{Y})$ in $\A$ is called \emph{hereditary} \cite{gj11}  if $\mathcal{X}$ is closed under taking kernels of admissible epimorphisms between objects of
$\mathcal{X}$ and if $\mathcal{Y}$ is closed under taking cokernels of admissible monomorphisms between objects of $\mathcal{Y}$. In this case, we say $\mathcal{X}$ is \emph{resolving}, and $\mathcal{Y}$ is \emph{coresolving}.

The following lemma is essentially taken from \cite[Lemma 2.3]{gj16}, where a variation of it
appears.

\begin{lem}\label{lem:hereditary-cotorsion-pair} Let $(\mathcal{X}, \mathcal{Y})$ be a
cotorsion pair in $\mathcal{A}$. If $\mathcal{A}$ has enough projective objects or enough injective objects, then $(\mathcal{X}, \mathcal{Y})$ is hereditary if and only if $\Ext_{\mathcal{A}}^{2}(X,Y)=0$ for every $X\in{\mathcal{X}}$, $Y\in{\mathcal{Y}}$.
\end{lem}

\subsection{Exact model structures}

Recall from \cite{gj11} that an exact model structure on a WIC exact category $\mathcal{A}$ is a model structure in the sense of \cite[Definition 1.1.3]{hv99} in which each of the following holds.
\begin{enumerate}
\item A map is a (trivial) cofibration if and only if it is an admissible monomorphism with a (trivially) cofibrant cokernel.

\item A map is a (trivial) fibration if and only if it is an admissible epimorphism with a (trivially) fibrant kernel.
\end{enumerate}

Gillespie showed in \cite{gj11} that the correspondence between model structures and cotorsion pairs from \cite{ho02} carries over to the case of WIC exact categories as below.

\begin{thm}\cite[Corollary 3.4]{gj11}\label{thm:2.4} Let $\A$ be a WIC exact category. There is a one-to-one correspondence between exact model structures on $\mathcal{A}$ and complete cotorsion pairs $\mathcal{(Q, R \cap W)}$ and $\mathcal{(Q \cap W, R)}$ where $\mathcal{W}$ is a thick subcategory of $\mathcal{A}$. Given a model structure, $\mathcal{Q}$ is the class of cofibrant objects, $\mathcal{R}$ the class of fibrant objects and $\mathcal{W}$ the class of trivial objects. Conversely, given the cotorsion pairs with $\mathcal{W}$ thick, a cofibration (resp., trivial cofibration) is an admissible monomorphism with a cokernel in $\mathcal{Q}$ (resp., $\mathcal{Q \cap W}$), and a fibration (resp., trivial fibration) is an admissible epimorphism with a kernel in $\mathcal{R}$ (resp., $\mathcal{R \cap W}$).
\end{thm}

Due to the above Hovey's one-to-one correspondence, we will often not distinguish between the Hovey triple and the actual model structure on a WIC exact category $\mathcal{A}$. For example, we may say that $\mathcal{M =(Q, W, R)}$
is an exact model structure and understand this to mean the model structure
associated to the Hovey triple $\mathcal{(Q \cap W, R)}$ on $\mathcal{A}$. On the other hand, we may say that an exact model structure is \emph{hereditary} if its associated Hovey triple is hereditary.

Let $\mathcal{W}$ be the class of weak equivalences. The homotopy category of the model category is the localization $\mathcal{C}[\mathcal{W}^{-1}]$ and is denoted by $\mathrm{Ho}(\mathcal{M})$.
By \cite[Section 4.2]{gj16b}, we know that if $\mathcal{M=(Q,W,R)}$ is a hereditary Hovey triple, then $\mathrm{Ho}(\mathcal{M})$ is a triangulated category and it is triangle equivalent to the stable category $(\mathcal{Q}\cap \mathcal{R})/\omega$, where $\omega=\mathcal{Q}\cap\mathcal{W}\cap \mathcal{R}$ is the class of projective-injective objects.

\subsection{Recollements of triangulated categories} Loosely, a recollement is an ``attachment'' of two triangulated categories. The standard reference is \cite{BBD}. Let $\mathcal{T}',~ \mathcal{T},~ \mathcal{T}''$ be triangulated categories. We give the definition that appeared in \cite{he05} based on localization and colocalization sequences.

\begin{definition} Let $\mathcal{T}'\stackrel{F}\rightarrow \mathcal{T}\stackrel{G}\rightarrow \mathcal{T}''$ be a sequence of triangulated functors between triangulated categories. We say it is a \emph{localization sequence} when there exist right adjoints $F_{\rho}$ and $G_{\rho}$ giving a diagram of functors as below with the listed properties.

$$\xymatrix@C=50pt@R=50pt{\ \mathcal{T}'\ar[r]^{F}&\ \mathcal{T}\ar@/^1pc/[l]^{F_{\rho}}\ar[r]^{G} &\ \mathcal{T}''\ar@/^1pc/[l]^{G_{\rho}}}$$

\begin{enumerate}
\item The right adjoint $F_{\rho}$ of $F$ satisfies $F_{\rho}\circ F=1_{\mathcal{T}'}$.

\item The right adjoint $G_{\rho}$ of $G$ satisfies $G\circ G_{\rho} =1_{\mathcal{T}''}$.

\item For any object $X\in \mathcal{T}$, we have $GX=0$ if and only if $X\cong FX'$ for some $X'\in \mathcal{T}'$.
\end{enumerate}

A colocalization sequence is the dual. That is, there must exist left adjoints $F_{\lambda}$ and $G_{\lambda}$ with the analogous properties.
\end{definition}
It is true that if $\mathcal{T}'\stackrel{F}\rightarrow \mathcal{T}\stackrel{G}\rightarrow \mathcal{T}''$ is a localization sequence then $\mathcal{T}''\stackrel{G_{\rho}}\rightarrow \mathcal{T}\stackrel{F_{\rho}}\rightarrow \mathcal{T}'$ is a colocalization sequence and if $\mathcal{T}'\stackrel{F}\rightarrow \mathcal{T}\stackrel{G}\rightarrow \mathcal{T}''$  is a colocalization sequence then $\mathcal{T}''\stackrel{G_{\lambda}}\rightarrow \mathcal{T}\stackrel{F_{\lambda}}\rightarrow \mathcal{T}'$ is a localization sequence.
This brings us to the definition of a recollement where the sequence of functors $\mathcal{T}'\stackrel{F}\rightarrow \mathcal{T}\stackrel{G}\rightarrow \mathcal{T}''$ is both a localization sequence and a colocalization sequence.

 \begin{definition} Let $\mathcal{T}'\stackrel{F}\rightarrow \mathcal{T}\stackrel{G}\rightarrow \mathcal{T}''$ be a sequence of exact functors between triangulated categories. We say $\mathcal{T}'\stackrel{F}\rightarrow \mathcal{T}\stackrel{G}\rightarrow \mathcal{T}''$ induces a \emph{recollement} if it is both a localization sequence and a colocalization sequence as shown in the picture
$$\xymatrix@C=50pt@R=50pt{\ \mathcal{T}'\ar[r]^{F}&\ \mathcal{T}\ar@/^1pc/[l]^{F_{\rho}}\ar@/_1pc/[l]_{F_{\lambda}}\ar[r]^{G} &\ \mathcal{T}''\ar@/^1pc/[l]^{G_{\rho}}\ar@/_1pc/[l]_{G_{\lambda}}}$$
\end{definition}

So the idea is that a recollement is a colocalization sequence ``glued'' with a localization sequence.

\subsection{Recollements of exact categories} In this subsection, we recall the definition of a  recollement situation in the context of exact categories (see \cite{Wand-wei-Zhang}). For an additive functor $F:\mathcal{A}\to \mathcal{C}$ between additive categories, we denote by $\textrm{im}F=\{C\in {\mathcal{C}} \ | \ C\cong F(A)\ \textrm{for} \ \textrm{some} \ A\in{\mathcal{A}}\}$ the essential image of $F$ and by $\textrm{ker}F=\{A\in\mathcal{A} \ | \ F(A)=0\}$ the kernel of $F$.

\begin{definition}\cite[Definition 3.1]{Wand-wei-Zhang}\label{df:recollements-exact}
Let $\mathcal{A},\mathcal{B},\mathcal{C}$ be three WIC exact categories.  A recollement of $\C$
relative to $\A$ and $\B$, denoted by $(\mathcal{A},\mathcal{C},\mathcal{B})$, is a diagram
$$\xymatrix@C=50pt@R=50pt{\ \mathcal{A}\ar[r]^{i_*}&\ \mathcal{C}\ar@/^1pc/[l]^{i^!}\ar@/_1pc/[l]_{i^*}\ar[r]^{j^*} &\ \mathcal{B}\ar@/^1pc/[l]^{j_*}\ar@/_1pc/[l]_{j_!}}$$
given by two exact functors $i_*$ and $j^*$, two right exact functors $i^*$, $j_!$ and two left exact functors $i^!$, $j_*$, which satisfies the following conditions:
 \begin{enumerate}
\item $(i^*,i_*)$, $(i_*,i^!)$, $(j_!,j^*)$ and $(j^*,j_*)$ are adjoint pairs;

\item $i_*$, $j_!$ and $j_*$ are fully faithful;

\item ${\rm im}i_*={\rm ker}j^*$;

\item For any $C\in\mathcal{C}$, there exists an exact sequence in $\C$
\begin{center}
{$ i_{*}i^{!}(C)\s{\sigma_{C}} \mto C \stackrel{\eta_{C}} \to j_{*}j^{*}(C)\eto i_{*}(A)$}
\end{center}
with $A\in{\A}$, where $\sigma_{C}$ and $\eta_{C}$ are given by the adjunction morphisms;

\item For any $C\in\mathcal{C}$, there exists an exact sequence in $\C$
\begin{center}
{$i_{*}(A')\mto j_{!}j^{*}(C)\stackrel{\varepsilon_{C}} \to C\s{\delta_{C}}\eto i_{*}i^{*}(C)$}
\end{center}
with $A'\in{\A}$, where $\varepsilon_{C}$ and $\delta_{C}$ are given by the adjunction morphisms.
 \end{enumerate}
 
 In this case one says that $(\A,\C,\B)$ is a \emph{recollement of exact categories}.
 \end{definition}

 If the categories $\mathcal{A},\mathcal{B}$ and $\mathcal{C}$ are abelian, then Definition \ref{df:recollements-exact} coincides with the definition of recollement of abelian categories.
We refer to \cite[Section 2.1]{Psaroudakis} for examples of recollements of abelian categories. For examples of recollement of exact categories, we refer to Section 3.

\vspace{2mm}
\textbf{Notation for units and counits}. Throughout, we denote by $\delta:1_{\C}\to i_{*}i^{*} $ (resp., $\eta:1_{\C}\to j_{*}j^{*} $), the unit of the adjoint pair $(i^{*},i_{*})$ (resp., $(j^{*},j_{*})$), and by $\sigma:i_{*}i^{!}\to 1_{\C}$ (resp., $\varepsilon: j_{!} j^{*}\to 1_{\C}$), the counit of the adjoint pair $(i_{*},i^{!})$ (resp., $(j_{!},j^{*})$).

\vspace{2mm}
We list some properties of recollements (see \cite[Lemma 3.3]{Wand-wei-Zhang}), which will be used in the sequel.

\begin{lem}\label{lem:property-of-recollement1} The following are true for any recollement $(\mathcal{A},\mathcal{C},\mathcal{B})$ of exact categories.
\begin{enumerate}
\item $i^*j_{!}=0=i^{!}j_{*}$.
\item All the natural transformations
$$i^{*}i_{*}\rightarrow 1_{\mathcal{A}}, \ 1_{\mathcal{A}}\to i^{!}i_{*}, \ 1_{\mathcal{B}}\to j^{*}j_{!}, j^{*}j_{*}\rightarrow 1_{\mathcal{B}}$$
are natural isomorphisms.
\item  $i^*$ preserves projective objects and $i^!$ preserves injective objects.
\item $j_!$ preserves projective objects and $j_*$ preserves injective objects.
\item If $i^!$ is exact, then $j_*$ is exact.
\end{enumerate}
\end{lem}

In the following sections, we always assume that $(\mathcal{A},\mathcal{C},\mathcal{B})$ is a recollement of exact categories defined in Definition \ref{df:recollements-exact}, where $\mathcal{A},\mathcal{B},\mathcal{C}$ are WIC exact categories.

\section{Constructing recollements of exact categories from recollements of abelian categories}
We begin this section with the following easy observation.

\begin{lem}\label{lem:2.2'} Let $C$ be an object in $\C$.
\begin{enumerate}
\item If $i^{*}j_{*}j^{*}(C)=0$, then $\eta_{C}: C\to j_{*} j^{*}(C)$ is an admissible epimorphism. Thus, there exists an admissible exact sequence $ i_{*}i^{!}(X)\s{\sigma_{X}} \mto X \stackrel{\eta_{X}} \eto j_{*}j^{*}(X)$ in $\C$.
\item If $i^{!}j_{!}j^{*}(C)=0$, then $\varepsilon_{C}: j_{!} j^{*}(C)  \to C$ is an admissible  monomorphism. Thus, there exists an admissible exact sequence $j_{!}j^{*}(X)\stackrel{\varepsilon_{X}} \mto X\s{\delta_{X}}\eto i_{*}i^{*}(X)$ in $\C$.
\end{enumerate}
\end{lem}
\begin{proof} We only prove (1), and the proof of (2) is similar. Note that there exists an exact sequence
$ i_{*}i^{!}(C)\s{\sigma_{C}} \mto C \stackrel{\eta_{C}} \to j_{*}j^{*}(C)\s{f}\eto i_{*}(A)$ in $\C$
with $A\in{\A}$. Hence $i^{*}(f):i^{*}j_{*}j^{*}(C)\to i^{*}i_{*}(A)$ is an admissible epimorphism in $\A$. Since $i^{*}j_{*}j^{*}(C)=0$ by hypothesis, it follows that $A\cong i^{*}i_{*}(A)=0$. So $\eta_{C}: C\to j_{*} j^{*}(C)$ is an admissible epimorphism, as desired.
\end{proof}

\begin{lem}\label{lem:2.2''} Let $M_1\stackrel{f}\mto M_2 \stackrel{g}\eto M_3$ be an admissible exact sequence in $\C$.
 \begin{enumerate}
\item If $\eta_{M_{i}}: M_{i}\to j_{*} j^{*}(M_{i})$ is admissible epic for $i=1,2,3$, then $i^{!}(M_1)\stackrel{i^{!}(f)}\mto i^{!}(M_2) \stackrel{i^{!}(g)}\eto i^{!}(M_3)$ is an admissible exact sequence in $\mathcal{A}$.
\item If $\varepsilon_{M_{i}}: j_{!} j^{*}(M_{i})  \to M_{i}$ is admissible monic for $i=1,2,3$, then $i^{*}(M_1)\stackrel{i^{*}(f)}\mto i^{*}(M_2) \stackrel{i^{*}(g)}\eto i^{*}(M_3)$ is an admissible exact sequence in $\mathcal{A}$.
\end{enumerate}
\end{lem}
\begin{proof}  We only prove (1), and the proof of (2) is similar. Note that we have a left exact sequence
$i^{!}(M_1)\stackrel{i^{!}(f)}\mto i^{!}(M_2)\stackrel{i^{!}(g)} \to i^{!}(M_3)$ in $\mathcal{A}$. If we set $N:=\textrm{coker}(i^{!}(f))$, to prove the exactness of the left sequence $i^{!}(M_1)\stackrel{i^{!}(f)}\mto i^{!}(M_2) \stackrel{i^{!}(g)}\to i^{!}(M_3)$,
 it suffices to show that $i^{!}(M_3)\cong N$. By hypothesis, we have the following commutative diagram
$$\xymatrix{
i_{*}i^{!}(M_1) \ar@{>->}[r]\ar@{>->}[d]&M_1 \ar@{>->}[d]\ar@{->>}[r]^{\eta_{M_1}}&j_{*}j^{*}(M_1)\ar@{>->}[d]\\
i_{*}i^{!}(M_2) \ar@{>->}[r]\ar[d]&M_2 \ar@{>->}[d]^{g}\ar@{->>}[r]^{\eta_{M_2}}&j_{*}j^{*}{(M_2)}\ar[d]^{j_{*}j^{*}{(g)}}\\
i_{*}i^{!}(M_3) \ar@{>->}[r]&M_3\ar@{->>}[r]^{\eta_{M_3}}&j_{*}j^{*}(M_3),\\}$$
where all rows are admissible exact sequences.
Since $j_{*}j^{*}{(g)}\eta_{M_2}=\eta_{M_3}g$ is admissible epic, so is $j_{*}j^{*}{(g)}$. Hence the sequence $i_{*}i^{!}(M_1) \mto i_{*}i^{!}(M_2)\to i_{*}i^{!}(M_3)$ in the above diagram is an admissible exact sequence by \cite[Corollary 8.13]{TBh10}. Since $i_{*}$ is an exact functor, $i_{*}i^{!}(M_1) \mto i_{*}i^{!}(M_2)\eto i_{*}(N)$ is an admissible exact sequence in $\mathcal{C}$. Thus $i_{*}(N)\cong i_{*}i^{!}{(M_3)}$, and so $N\cong i^{*}i_{*}{(N)}\cong i^{*}i_{*}i^{!}{(M_3)}\cong i^{!}(M_3)$, as desired.
\end{proof}

\begin{prop}\label{prop:2.5} Let $(\mathcal{A},\mathcal{C},\mathcal{B})$ be a recollement of exact categories.
 \begin{enumerate}
\item The following conditions are equivalent.
 \begin{enumerate}
\item $i^{!}$ is an exact functor.

\item $i^{*}j_{*}=0$.
\item $i^{*}j_{*}j^{*}=0$.
\end{enumerate}

\item The following conditions are equivalent.
 \begin{enumerate}
\item $i^{*}$ is an exact functor.
\item $i^{!}j_{!}=0$.
\item $i^{!}j_{!}j^{*}=0$.
\end{enumerate}
 \end{enumerate}
\end{prop}
\begin{proof}  We only prove (1), and the proof of (2) is similar.

$(a)\Rightarrow(b)$. The proof is model on that of Proposition 8.8 in \cite{FP2004}. Let $X$ be an object in $\B$. It follows that $\delta_{j_{*}(X)}:j_{*}(X)\eto i_{*}i^{*}j_{*}(X)$ is an admissible epimorphism. Since $i^{!}$ is an exact functor, $i^{!}(\delta_{j_{*}(X)}):i^{!}j_{*}(X)\eto i^{!}i_{*}i^{*}j_{*}(X)$ is also admissible epic. Note that $i^{!}j_{*}=0$ by Lemma \ref{lem:property-of-recollement1}(1). Thus $i^{*}j_{*}(X)\cong{i^{!}i_{*}i^{*}j_{*}(X)}=0$, as desired.

$(b)\Rightarrow(c)$ is trivial.

$(c)\Rightarrow(a)$ holds by Lemmas \ref{lem:2.2'} and \ref{lem:2.2''}.
\end{proof}

\begin{rem}\label{remark:2.6} We note that Proposition \ref{prop:2.5} not only generalizes \cite[Proposition 8.8]{FP2004} from recollements of abelian categories to the setting of exact categores, but also refines it by deleting two superfluous assumptions ``with enough projectives" and ``with enough injectives". Also, our proof here is different from that in \cite{FP2004}.
\end{rem}

We are now in a position to state and prove the main result of this section, which provides a method to construct recollements of exact categories from recollements of abelian categories.
\begin{thm}\label{thm:3.5} Given the following recollement of abelian categories
$$\xymatrix@C=50pt@R=50pt{\ \mathcal{A}\ar[r]^{i_*}&\ \mathcal{C}\ar@/^1pc/[l]^{i^!}\ar@/_1pc/[l]_{i^*}\ar[r]^{j^*} &\ \mathcal{B},\ar@/^1pc/[l]^{j_*}\ar@/_1pc/[l]_{j_!} }\eqno{(3.1)}$$
we set $\mathcal{C}_{1}:=\{C\in {\mathcal{C}} \ | \ i^{*}j_{*} j^{*}(C)=0\}$ and $\mathcal{B}_1:=\{B\in{\B} \ | \ B\cong j^{*}(C)  \ \textrm{for} \ \textrm{some} \ C\in{\mathcal{C}_1}\}$.
 \begin{enumerate}
 \item The follow are equivalent:
 \begin{enumerate}
 \item $i^{!}:\C\to\A$ is an exact functor;
 \item $\mathcal{C}_{1}=\C$;
 \item $\mathcal{B}_{1}=\B$.
  \end{enumerate}
\item If $i^{!}:\C\to\A$ is not an exact functor, then the recollement (3.1) can induce the following recollement of exact categories:

$$\xymatrix@C=50pt@R=50pt{\ \mathcal{A}\ar[r]^{i_*}&\ \mathcal{C}_{1}\ar@/^1pc/[l]^{i^!}\ar@/_1pc/[l]_{i^*}\ar[r]^{j^*} &\ \mathcal{B}_{1}\ar@/^1pc/[l]^{j_*}\ar@/_1pc/[l]_{j_!}}\eqno{(3.2)}$$

\noindent such that both $i^!:\C_{1}\to\A$ and $j_*:\B_{1}\to\C_{1}$ are exact functors.
 \end{enumerate}
\end{thm}

\begin{proof} (1) $(a)\Rightarrow(b)$ follows from Proposition \ref{prop:2.5}(1) and $(b)\Rightarrow(c)$ is trivial. To prove $(c)\Rightarrow(a)$, it suffices to show $i^{*}j_{*}(B)=0$ for any $B\in{\B}$ by Proposition \ref{prop:2.5}(1). Since $\mathcal{B}_{1}=\B$ by (c), there exists $C\in{\C_{1}}$ such that $B\cong{j^{*}(C)}$. So $i^{*}j_{*}(B)\cong i^{*}j_{*}j^{*}(C)=0$, as desired.

(2) We first claim that $\mathcal{B}_{1}$ and $\mathcal{C}_{1}$ are WIC exact categories with exact structures induced from $\B$ and $\C$, respectively. By the additive of the functors $i^{*},j_{*}$ and $j^{*}$, it suffices to check that $\mathcal{B}_{1}$ and $\mathcal{C}_{1}$ are closed under extensions in $\B$ and $\C$, respectively. Let $0\to N_1\stackrel{f}\to N_2 \stackrel{g}\to N_3\to 0$ be an exact sequence in $\B$ with $N_{1},N_{3}\in \mathcal{B}_{1}$. Then there exist $M_1,M_3\in{\mathcal{C}_{1}}$ such that $N_1\cong j^{*}(M_1)$ and $N_3\cong j^{*}(M_3)$. Thus there exists $M_2\in\mathcal{C}$ such that $N_2\cong j^{*}(M_2)$ by Lemma \ref{lem:property-of-recollement1}(3). Thanks to \cite[Proposition 4.3]{FP2004}, we have the following commutative diagram with exact rows and columns
$$\xymatrix{&&&0\ar[d]&\\
&&j_{!}(N_1)\ar[r]\ar[d]&j_{*}(N_1)\ar[d]&\\
0\ar[r]& i_{*}i^{!}j_{!}(N_2) \ar[r]\ar[d]&j_{!}(N_2) \ar[d]^{j_{!}(g)}\ar[r]&j_{*}(N_2)\ar[d]^{j_{*}(g)}\ar[r]& i_{*}i^{*}j_{*}(N_2)\\
0\ar[r]& i_{*}i^{!}j_{!}(N_3) \ar[r]&j_{!}(N_3) \ar[d]\ar[r]&j_{*}(N_3)\ar[r]&0\\
&&0.&&\\}$$
Since $j_{!}(g)$ is epic, so is $j_{*}(g)$. This implies that $0\to j_{*}j^{*}(M_1)\to j_{*}j^{*}(M_2) \to j_{*}j^{*}(M_3)\to 0$  is exact in $\mathcal{C}$. Applying the functor $i^{*}$, we have an exact sequence in $\A$
$$i^{*}j_{*}j^{*}(M_1)\to i^{*}j_{*}j^{*}(M_2) \to i^{*}j_{*}j^{*}(M_3)\to 0.$$
Since $M_{1},M_{3}\in \mathcal{C}_{1}$, we have $i^{*}j_{*}j^{*}(M_2)=0$. Therefore, $M_2\in \mathcal{C}_{1}$. So $\mathcal{B}_{1}$ is closed under extensions in $\B$.

To prove that $\mathcal{C}_{1}$ is closed under extensions in $\C$, we consider an exact sequence $0\to X_1\to X_2 \to X_3\to 0$ in $\C$ with $X_1, X_3 \in \mathcal{C}_{1}$. Applying the functor $j^{*}$, we obtain an exact sequence $0\to j^{*}(X_1)\to j^{*}(X_2) \to j^{*}(X_3)\to 0$. Note that $j^{*}(X_1), j^{*}(X_3)\in \mathcal{B}_{1}$. Since $\mathcal{B}_{1}$ is  closed under extensions, there exists an object $Y\in \mathcal{C}_{1}$ such that $j^{*}(X_2)\cong j^{*}(Y)\in \mathcal{B}_{1}$. Thus  $i^{*}j_{*}j^{*}(X_2)\cong i^{*}j_{*}j^{*}(Y)=0$, so $X_2\in \mathcal{C}_{1}$, as desired.

Next we claim that (3.2) is a recollement of exact categories. By Lemma \ref{lem:2.2'}(1) and \cite[Proposition 2.6(ii)]{Psaroudakis}, it suffices to show the condition (5) of Definition \ref{df:recollements-exact}. Let $C$ be an object in $\C_{1}$. Note that there exists an exact sequence in $\C$
\begin{equation*}\label{equ:3.1}
{0\to i_{*}(A) \to j_{!}j^{*}(C)\stackrel{\varepsilon_{C}} \to C\to  i_{*}i^{*}(C)\to 0}\eqno{(3.3)}
 \end{equation*}
with  $A\in{\mathcal{A}}$. It is easy to check that each term in the above exact sequence belongs to $\C_{1}$. If we set $K:=\textrm{im}(\varepsilon_{C})$, we need to show that $K$ is in $\C_{1}$. Applying $j^{*}$ to the sequence (3.3), it follows that $j^{*}(K)\cong {j^{*}(C)}$. Thus  $i^{*}j_{*}j^{*}(K)\cong i^{*}j_{*}j^{*}(C)=0$, so $K\in \mathcal{C}_{1}$.

Finally, it follows from Proposition \ref{prop:2.5}(1) that $i^{!}: \mathcal{C}_{1}\to \mathcal{A}$ is exact. So $j_{*}: \mathcal{B}_{1}\to \mathcal{C}$ is also exact by Lemma \ref{lem:property-of-recollement1}(5). This completes the proof.
\end{proof}

The following example, due to Zhang-Cui-Rong \cite{ZH}, shows the recollement of abelian categories can induce a nontrivial recollement of exact categories provided that $i^{!}$ is not an exact functor.

\begin{ex}\label{ex:3.1} Let $A=B$ be the path algebra $k(1 \rightarrow 2)$, where $\mathrm{char} k \neq 2$. Write the
conjunction of paths from right to left. Thus $e_{1}Ae_{2} = 0$ and $e_{2}Ae_{1} \cong k$. Take $M = N =
Ae_{2} \otimes_{k} e_{1}A$. Then $M \otimes_{A} N = 0 = N \otimes_{A} M$. Let $\Lambda$ be the Morita ring $\left(\begin{smallmatrix}  A & N \\  N & A \\\end{smallmatrix}\right)$. By \cite[Section 2.4]{ZH}, we obtain the recollement
$$\xymatrix@C=50pt@R=50pt{\ {\rm Mod}\textrm{-}A\ar[r]^{i_{*}}&\ {\rm Mod}\textrm{-}\Lambda\ar@/^1pc/[l]^{i^{!}}\ar@/_1pc/[l]_{i^{*}}\ar[r]^{j^{*}} &\ {\rm Mod}\textrm{-}A,\ar@/^1pc/[l]^{j_{*}}\ar@/_1pc/[l]_{j_{!}} }\eqno{(3.4)}$$
where $i^{*}$ is given by $\left(\begin{smallmatrix}  X  \\   Y \\\end{smallmatrix}\right)_{f,g}\mapsto \mathrm{coker}g$; $i_{*}$ is given by $X\mapsto \left(\begin{smallmatrix}  X  \\   0 \\\end{smallmatrix}\right)_{0,0}$; $i^{!}$ is given by $\left(\begin{smallmatrix}  X  \\   Y \\\end{smallmatrix}\right)_{f,g}\mapsto \mathrm{Ker}\widetilde{f}$, where $\widetilde{f}:=\eta_{X,Y}(f)$ and $\eta_{X,Y}$ is the adjunction isomorphism
$\mathrm{Hom}_{A}(N\otimes_{A}X, Y )\cong
\mathrm{Hom}_{A}(X, \mathrm{Hom}_{A}(N, Y ))$; $j_{!}$ is given by $Y\mapsto \left(\begin{smallmatrix}  N\otimes_{A}Y  \\   Y \\\end{smallmatrix}\right)_{0,1}$; $j^{*}$ is given by $\left(\begin{smallmatrix}  X  \\   Y \\\end{smallmatrix}\right)_{f,g}\mapsto Y$; $j_{*}$ is given by $Y\mapsto \left(\begin{smallmatrix}  \mathrm{Hom}_{A}(N,Y)  \\   Y \\\end{smallmatrix}\right)_{\epsilon_{Y},0}$, where $\epsilon$ is the counit
$ N \otimes_{A}\mathrm{Hom}_{A}(N,-) \rightarrow\mathrm{1}_{\text{Mod-}A}$. In this case, we have $\mathcal{B}_{1}=\{Y\in\text{Mod-}A\mid \mathrm{Hom}_{A}(N,Y)=0\}$.

The Auslander-Reiten quiver $\Gamma(\mathrm{mod}\text{-}A)$ of the module category $\text{mod-}A$ has the form
$$\xymatrix@M=1pt@!0{
   &   A{e_{1}} \ar[dr]^{\pi}\\
   S_{2}\ar[ur]^{\sigma}&& S_{1}. }$$
Since $_{A}N$ is isomorphic to the simple left $A$-module $Ae_{2} = S_2$, it follows that $$\mathcal{B}_{1}=\{Y\in\text{ Mod-}A\mid \mathrm{Hom}_{A}(N,Y)=0\}=\mathrm{Add}(S_1).$$
Thus we have
\begin{align*}
   \mathcal{C}_{1}&=\{\left(\begin{smallmatrix}  X  \\   Y \\\end{smallmatrix}\right)_{f,g}\in\text{Mod-}\Lambda\mid i^{*}j_{*}j^{*}(\left(\begin{smallmatrix}  X  \\   Y \\\end{smallmatrix}\right)_{f,g})=0\}\\
   &=\{\left(\begin{smallmatrix}  X  \\   Y \\\end{smallmatrix}\right)_{f,g}\in\text{Mod-}\Lambda\mid \mathrm{Hom}_{A}(N,Y)=0\}
   =\{\left(\begin{smallmatrix}  X  \\   Y \\\end{smallmatrix}\right)_{f,g}\in\text{Mod-}\Lambda\mid Y\in{\B}_{1}\}.
 \end{align*}
This informs us that ${\A=\text{Mod-}A}\subsetneqq{\mathcal{C}_{1}}\subsetneqq {\text{Mod-}\Lambda=\mathcal{C}}$ and $0\neq\mathcal{B}_{1}\subsetneqq {\text{Mod-}A=\mathcal{B}}$, as desired.
\end{ex}

\section{Proofs of the main results}
In this section, we prove the main results mentioned in the introduction. We keep the notation
introduced in the previous sections.

\subsection{Proof of Theorem \ref{thm:1.1}}\label{cotorsion-pair}

In what follows, we always assume that $\mathcal{U}',\mathcal{V}'$ are subcategories of $\A$ and $\mathcal{U}'',\mathcal{V}''$ are subcategories of $\B$.
Denote by $\mathcal{U}=\{C\in{\C} \ | \ i^{*}(C)\in{\mathcal{U}}', j^{*}(C)\in{\mathcal{U}}'', \ \varepsilon_{M}: j_{!} j^{*}(M)  \to M \ \textrm{is} \ \textrm{an} \ \textrm{admissible} \ \textrm{monomorphism}\}$ and by $\mathcal{V}=\{C\in{\C} \ | \ i^{!}(C)\in{\mathcal{V}}', j^{*}(C)\in{\mathcal{V}}''\}$.

The following result is crucial to the proof of Theorem \ref{thm:1.1}.

\begin{prop}\label{thm:2.9}  Assume that the WIC exact category category $\C$ in the recollement (1.1) has enough projective and injective objects. If $i^!$ is an exact functor and $j_{!}$ is $^{\perp}(\mathcal{V}'')$-exact, then $(\mathcal{U}',\mathcal{V}')$ and $(\mathcal{U}'',\mathcal{V}'')$ are hereditary cotorsion pairs in $\mathcal{A}$ and $\mathcal{B}$, respectively if and only if $(\mathcal{U},\mathcal{V})$ is a hereditary cotorsion pair in $\C$.
\end{prop}
To prove Proposition \ref{thm:2.9}, we need some preparations.

\begin{lem}\label{lem:2.4} Let $i_{*}(A)\mto M\eto N$ be an admissible exact sequence in $\C$. If $\varepsilon_{N}:j_{!}j^{*}(N)\to N$ is admissible monic, then $i^{*}i_{*}(A)\mto i^{*}(M)\eto i^{*}(N)$ is an admissible exact sequence in $\A$.
\end{lem}
\begin{proof} Since $i^{*}$ is right exact,  we have a right exact sequence $i^{*}i_{*}(A)\to i^{*}(M)\eto i^{*}(N)$ in $\A$. If we set $K:=\textrm{ker}(i^{*}(M)\to i^{*}(N))$, we only need to show that $i^{*}i_{*}(A)\cong K$. Note that we have the following commutative diagram
$$\xymatrix{
j_{!}j^{*}i_{*}(A) \ar[r]\ar[d]&j_{!}j^{*}(M) \ar[d]\ar@{->>}[r]&j_{!}j^{*}(N)\ar[d]^{\varepsilon_{N}}\\
i_{*}(A) \ar@{>->}[r]\ar@{->>}[d]&M\ar@{->>}[d]\ar@{->>}[r]&N\ar@{->>}[d]\\
i_{*}i^{*}i_{*}(A)\ar[r]& i_{*}i^{*}(M)\ar@{->>}[r]&i_{*}i^{*}(N)\\}$$
such that all columns, and both the first and third rows are right exact. Since $\varepsilon_{N}:j_{!}j^{*}(N)\to N$ is an admissible monomorphism, it follows from the snake lemma (see \cite[Exercise 8.15]{TBh10}) that the right exact sequence $i_{*}i^{*}i_{*}(A)\to i_{*}i^{*}(M)\eto i_{*}i^{*}(N)$ in the above commutative diagram is admissible exact. Since $ i_{*}(K)\mto i_{*}i^{*}(M)\eto i_{*}i^{*}(N)$ is also an admissible exact sequence in $\C$, we have $i_{*}(K)\cong i_{*}i^{*}i_{*}(A)$. So $K\cong i^{*}i_{*}(A)$ by noting that $i_{*}$ is fully faithful. This completes the proof.
\end{proof}

\begin{lem}\label{lem:2.6} If $i^!$ is an exact functor, then any object $M\in{\C}$ gives the following admissible exact sequence $$\xymatrix@C=2cm@R1cm{ i_{*}i^{!}j_{!}j^{*}(M) \ar@{>->}[r]^{\left(\begin{smallmatrix}\sigma_{j_{!}j^{*}(M)}\\i_{*}i^{!}(\varepsilon_{M})\end{smallmatrix}\right)}&
j_{!}j^{*}(M)\oplus  i_{*}i^{!}(M) \ar@{->>}[r]^{(-\varepsilon_{M},\sigma_{M})} &M.}$$
\end{lem}
\begin{proof} Note that $i^{*}j_{*}j^{*}(M)=0$ by Proposition \ref{prop:2.5}(1). It follows from Lemma \ref{lem:2.2'} that $\eta_{M}:M\to j_{*} j^{*}(M)$ is an admissible epimorphism. Thus we have the following commutative diagram
$$\xymatrix{
i_{*}i^{!}j_{!}j^{*}(M) \ar@{>->}[r]^{\sigma_{j_{!}j^{*}(M)}}\ar[d]^{i_{*}i^{!}(\varepsilon_{M})}&j_{!}j^{*}(M) \ar[d]^{\varepsilon_{M}}\ar@{->>}[r]&j_{*}j^{*}(M)\ar@{=}[d]\\
i_{*}i^{!}(M) \ar@{>->}[r]^{\sigma_{M}}&M \ar@{->>}[r]^{\eta_{M}}&j_{*}j^{*}(M).\\}$$
Thanks to \cite[Proposition 2.12]{TBh10}, we have the desired admissible exact sequence.
\end{proof}

\begin{prop}\label{prop:2.8} Assume that $\C$ has enough projective and injective objects and $M$ is an object in $\C$. If $i^!$ is an exact functor, then $\Ext_{\C}^{1}(M,i_{*}(I))=0$ for any injective object $I$ in $\A$ if and only if $\varepsilon_{M}: j_{!} j^{*}(M)  \to M$ is an admissible monomorphism.
\end{prop}
\begin{proof} Since $\C$ has enough projective and injective objects by hypothesis, it follows from \cite[Lemma 3.3(5)]{Wand-wei-Zhang} that $\A$ has enough projective and injective objects. By Lemma \ref{lem:2.6}, we have the following admissible exact sequence
$$\xymatrix@C=2cm@R1cm{ i_{*}i^{!}j_{!}j^{*}(M) \ar@{>->}[r]^{\left(\begin{smallmatrix}\sigma_{j_{!}j^{*}(M)}\\i_{*}i^{!}(\varepsilon_{M})\end{smallmatrix}\right)}&
j_{!}j^{*}(M)\oplus  i_{*}i^{!}(M) \ar@{->>}[r]^{(-\varepsilon_{M},\sigma_{M})} &M.}$$

``$\Rightarrow$". Let $f:i^{!}j_{!}j^{*}(M)\to I$ be an admissible monomorphism in $\A$ with $I$ injective. Since $\Ext_{\C}^{1}(M,i_{*}(I))=0$ by hypothesis, there exists a morphism $(g,h):j_{!}j^{*}(M)\oplus  i_{*}i^{!}(M)\to i_{*}(I)$ such that $$(g,h)\left(\begin{smallmatrix}\sigma\\i_{*}i^{!}(\varepsilon_{M})\end{smallmatrix}\right)=i_{*}(f).$$ Thus $g\sigma+h{i_{*}i^{!}(\varepsilon_{M})}=i_{*}(f)$. Since $g\in{\Hom_{\C}(j_{!}j^{*}(M),i_{*}(I))\cong \Hom_{\C}(i^{*}j_{!}j^{*}(M),I)=0}$, we have $i_{*}(f)=h{i_{*}i^{!}(\varepsilon_{M})}$. Note that $i_{*}(f)$ is admissible monic. Thus  $i_{*}i^{!}(\varepsilon_{M})$ is also admissible monic. So $\varepsilon_{M}: j_{!} j^{*}(M)  \to M$ is an admissible monomorphism by the commutative diagram in Lemma \ref{lem:2.6}.

``$\Leftarrow$". Let $I$ be an injective object in $\A$. Then we have the following exact sequence
$$\xymatrix@C=1.2cm{\Hom_{\C}(j_{!}j^{*}(M)\oplus  i_{*}i^{!}(M),i_{*}(I))\ar[r]& \Hom_{\C}(i_{*}i^{!}j_{!}j^{*}(M),i_{*}(I))\ar[r]& \\ \Ext_{\C}^{1}(M,i_{*}(I))\ar[r]&\Ext_{\C}^{1}(i_{*}i^{!}j_{!}j^{*}(M),i_{*}(I)).}$$
Note that $\Ext_{\C}^{1}(i_{*}i^{!}j_{!}j^{*}(M),i_{*}(I))\cong \Ext_{\A}^{1}(i^{!}j_{!}j^{*}(M),i^{!}i_{*}(I))\cong \Ext_{\A}^{1}(i^{!}j_{!}j^{*}(M),I)=0$ by \cite[Lemma 3.3(7)]{Wand-wei-Zhang}. To prove $\Ext_{\C}^{1}(M,i_{*}(I))=0$, it suffices to show that
$$\xymatrix@C=1.2cm{\Hom_{\C}({\left(\begin{smallmatrix}\sigma_{j_{!}j^{*}(M)}\\i_{*}i^{!}(\varepsilon_{M})\end{smallmatrix}\right)},i_{*}(I)):\Hom_{\C}(j_{!}j^{*}(M)\oplus  i_{*}i^{!}(M),i_{*}(I))\ar[r]& \Hom_{\C}(i_{*}i^{!}j_{!}j^{*}(M),i_{*}(I))}$$
is an epimorphism. Since $i^{*}j_{!}j^{*}(M)=0$, $i^{*}i_{*}i^{!}(M)\cong i^{!}(M)$ and $i^{*}i_{*}i^{!}j_{!}j^{*}(M)\cong i^{!}j_{!}j^{*}(M)$, we have the following commutative diagram
$$\xymatrix{
\Hom_{\C}(j_{!}j^{*}(M)\oplus  i_{*}i^{!}(M),i_{*}(I))\ar[r]\ar[d]^{\cong}& \Hom_{\C}(i_{*}i^{!}j_{!}j^{*}(M),i_{*}(I))\ar[d]^{\cong}\\
\Hom_{\A}(i^{*}j_{!}j^{*}(M)\oplus i^{*}i_{*}i^{!}(M),I)\ar[r]\ar[d]^{\cong}& \Hom_{\A}(i^{*}i_{*}i^{!}j_{!}j^{*}(M),I)\ar[d]^{\cong}\\
\Hom_{\A}(i^{!}(M),I)\ar[r]^{\Hom_{\A}(i^{!}(\varepsilon_{M}),I)}& \Hom_{\A}(i^{!}j_{!}j^{*}(M),I).}$$
Note that $\varepsilon_{M}: j_{!} j^{*}(M)  \to M$ is an admissible monomorphism by hypothesis. Then we have an admissible exact sequence $j_{!}j^{*}(M) \mto M \to i_{*} i^{*}(M)\eto$ in $\C$. Since $i^{!}(\varepsilon_{M}):i^{!}j_{!}j^{*}(M)\to i^{!}(M)$ is an admissible monomorphism, it follows that $\Hom_{\A}(i^{!}(\varepsilon_{M}),I):\Hom_{\A}(i^{!}(M),I)\to \Hom_{\A}(i^{!}j_{!} j^{*}(M),I)$ is an epimorphism. So $\Hom_{\C}({\left(\begin{smallmatrix}\sigma_{j_{!}j^{*}(M)}\\i_{*}i^{!}(\varepsilon_{M})\end{smallmatrix}\right)},i_{*}(I))$ is an epimorphism, as desired.
\end{proof}

\begin{lem}\label{lem:2.10} The following are true for any recollement $(\mathcal{A},\mathcal{C},\mathcal{B})$ of exact categories.
\begin{enumerate}
\item If $P$ is a projective object in $\B$, then $j_{!}(P)$ is a projective object in $\C$.

\item If $i^!$ is an exact functor and $Q$ is a projective object in $\A$, then $i_{*}(Q)$ is a projective object in $\C$.

\item If $H$ is a projective object in $\C$, then $i^{*}(H)$ is a projective object in $\A$.

\item If $i^!$ is an exact functor and $N$ is a projective object in $\C$, then $j^{*}(N)$ is a projective object in $\B$.
 \end{enumerate}
\end{lem}
\begin{proof} By \cite[Lemma 3.3(5)]{Wand-wei-Zhang}, we only prove (1).  Assume that $ M_1\mto M_2\eto M_3$ is an admissible exact sequence in $\C$. Let $P$ be a projective object in $\B$.
Thus we have the following commutative diagram
$$\xymatrix@=3.8em{ \Hom_{\C}(j_{!}(P),M_2)\ar[d]^{\cong} \ar[r] & \Hom_{\C}(j_{!}(P),M_3) \ar[d]^{\cong}\\
\Hom_{\B}(P,j^{*}(M_2)) \ar[r] & \Hom_{\A}(P,j^{*}(M_3)).}$$
Since $j^{*}: \mathcal{C}\to \mathcal{B}$ is an exact functor, it follows that $\Hom_{\B}(P,j^{*}(M_2)) \to \Hom_{\B}(P,j^{*}(M_3))$ is an epimorphism. So $\Hom_{\C}(j_{!}(P),M_2)\to \Hom_{\C}(j_{!}(P),M_3)$ is an epimorphism by the commutative diagram above, as desired.
\end{proof}

\begin{lem}\label{lem:2.11} Assume that $i^!$ is an exact functor and $N$ is an object in $\C$.
\begin{enumerate}
\item If $\A$ has enough projective objects and $M$ is an object in $\A$, then $\Ext_{\C}^{i}(i_{*}(M),N)\cong \Ext_{\A}^{i}(M,i^{!}(N))$ for any $i\geq1$.

\item If $\C$ has enough projective objects and $U$ is an object in $\mathcal{C}$, then $\Ext_{\C}^{i}(U,i_{*}(N))\cong \Ext_{\A}^{i}(i^{*}(U),N)$ and $\Ext_{\C}^{i}(U,j_{*}(N))\cong \Ext_{\B}^{i}(j^{*}(U),N)$ for any $i\geq1$.

\item Assume that $\B$ has enough projective objects and $\mathcal{L}$ is a resolving subcategory of $\B$ which contains projective objects. If $j_{!}$ is $\mathcal{L}$-exact, then $\Ext_{\C}^{i}(j_{!}(L),N)\cong \Ext_{\B}^{i}(L,j^{*}(N))$ for any object $L$ in $\mathcal{L}$ and all $i\geq1$.
 \end{enumerate}
\end{lem}
\begin{proof} (1) Let $\cdots\to P_{2}\to P_{1}\to P_{0}\eto M$ be an exact sequence in $\A$ with each $P_i$ projective. It follows from Lemma \ref{lem:2.10}(2) that $\cdots\to i_{*}(P_{2})\to i_{*}(P_{1})\to i_{*}(P_{0})\eto i_{*}(M)$ is an exact sequence in $\C$ with each $i_{*}(P_{i})$ projective. For any integer $i\geq1$, we have the following commutative diagram
$$\xymatrix@=3.8em{ \Hom_{\C}(i_{*}(P_{i-1}),N)\ar[d]^{\cong} \ar[r] & \Hom_{\C}(i_{*}(P_{i}),N) \ar[d]^{\cong}\ar[r] & \Hom_{\C}(i_{*}(P_{i+1}),N) \ar[d]^{\cong}\\
\Hom_{\A}(P_{i-1},i^{!}(N)) \ar[r] & \Hom_{\A}(P_{i},i^{!}(N)) \ar[r] & \Hom_{\A}(P_{i+1},i^{!}(N)),}$$
which implies $\Ext_{\C}^{i}(i_{*}(M),N)\cong \Ext_{\A}^{i}(M,i^{!}(N))$.

(2) The proof is similar to that of (1).

(3) Let $L$ be an object in $\mathcal{L}$. Then there exists an exact sequence $\cdots\to Q_{2}\to Q_{1}\to Q_{0}\eto L$ in $\B$ with each $Q_i$ projective. Note that $j_{!}$ is $\mathcal{L}$-exact by hypothesis. Since $\mathcal{L}$ is a resolving subcategory of $\B$ which contains projective objects, it follows that $\cdots\to j_{!}(Q_{2})\to j_{!}(Q_{1})\to j_{!}(Q_{0})\eto j_{!}(L)$ is an exact sequence in $\C$ with each $j_{!}(Q_i)$ projective by Lemma \ref{lem:2.10}(1). For any integer $i\geq1$, we have the following commutative diagram
$$\xymatrix@=3.8em{ \Hom_{\C}(j_{!}(Q_{i-1}),N)\ar[d]^{\cong} \ar[r] & \Hom_{\C}(j_{!}(Q_{i}),N) \ar[d]^{\cong}\ar[r] & \Hom_{\C}(j_{!}(Q_{i+1}),N) \ar[d]^{\cong}\\
\Hom_{\B}(Q_{i-1},j^{*}(N)) \ar[r] & \Hom_{\B}(Q_{i},j^{*}(N)) \ar[r] & \Hom_{\B}(Q_{i+1},j^{*}(N)),}$$
which implies that $\Ext_{\C}^{1}(j_{!}(L),N)\cong \Ext_{\B}^{1}(L,j^{*}(N))$.
\end{proof}

We are now in a position to prove Proposition \ref{thm:2.9}.

{\bf Proof of Proposition \ref{thm:2.9}}. Note that $\mathcal{C}$ has enough projective objects by hypothesis. Since $j_{*}:\B\to\C$ is an exact functor by Lemma \ref{lem:property-of-recollement1}(5), it follows from \cite[Lemma 3.3(6)]{Wand-wei-Zhang} that $\mathcal{B}$ has enough projective objects.

``$\Rightarrow$". Assume that $(\mathcal{U}',\mathcal{V}')$ and $(\mathcal{U}'',\mathcal{V}'')$ are hereditary cotorsion pairs in $\mathcal{A}$ and $\mathcal{B}$, respectively. For lucidity, we divide the proof into 3-steps.

\textbf{Step 1:} $\Ext_{\C}^{2}(U,V)=0$ for every $U\in{\mathcal{U}}$ and $V\in{\mathcal{V}}$.
Since $i^{*}(U)\in{\mathcal{U}}'$ and  $i^{!}(V)\in{\mathcal{V}}'$, it follows from Lemma \ref{lem:2.11}(1) that $\Ext_{\C}^{2}(i_{*}i^{*}(U),V)\cong{\Ext_{\A}^{2}(i^{*}(U),i^{!}(V))=0}$. Note that $j^{*}(U)\in{\mathcal{U}}''$ and $j^{*}(V)\in{\mathcal{V}}''$. Applying Lemma \ref{lem:2.11}(3), it follows that $\Ext_{\C}^{2}(j_{!}j^{*}(U),V)\cong{\Ext_{\B}^{2}(j^{*}(U),j^{*}(V))}=0$. Since $j_{!}j^{*}(U)\mto U\eto i_{*}i^{*}(U)$ is admissible exact in $\C$, this means $\Ext_{\C}^{2}(U,V)=0$, as desired.

\textbf{Step 2:} $\mathcal{U}^{\perp}\subseteq\mathcal{V}$. Let $N$ be an object in $\mathcal{U}^{\perp}$. To prove $N\in{\mathcal{V}}$, it suffices to show $i^{!}(N)\in{\mathcal{V}'}$ and $j^{*}(N)\in{\mathcal{V}''}$. Let $M$ be an object in $\mathcal{U}'$. Then $i_{*}(M)$ belongs to $\mathcal{U}$. It follows from Lemma \ref{lem:2.11}(1) that $\Ext_{\A}^{1}(M,i^{!}(N))\cong \Ext_{\C}^{1}(i_{*}(M),N)=0$. So $i^{!}(N)$ belongs to ${\mathcal{V}'}$. On the other hand, we assume that $L$ is an object in $\mathcal{U}''$. Then $j_{!}(L)$ is in $\mathcal{U}$. Applying Lemma \ref{lem:2.11}(3), we obtain $\Ext_{\B}^{1}(L,j^{*}(N))\cong \Ext_{\C}^{1}(j_{!}(L),N)=0$. So $j^{*}(N)$ belongs to ${\mathcal{V}''}$.

\textbf{Step 3:}  $^{\perp}\mathcal{V}\subseteq\mathcal{U}$. Let $M$ be an object in $^{\perp}\mathcal{V}$. To prove $M\in{\mathcal{U}}$, it suffices to show $i^{*}(M)\in{\mathcal{U}}'$, $j^{*}(M)\in{\mathcal{U}}''$ and $\varepsilon_{M}: j_{!} j^{*}(M)  \to M \ \textrm{is} \ \textrm{an} \ \textrm{admissible} \ \textrm{monomorphism}$. Let $N$ be an object in $\mathcal{V}'$. Then $i_{*}(N)$ is in $\mathcal{V}$. By Proposition \ref{prop:2.8}, one can check that $\varepsilon_{M}: j_{!} j^{*}(M)  \to M \ \textrm{is} \ \textrm{a} \ \textrm{admissible} \ \textrm{monomorphism}$. Thus we have an admissible exact sequence $ j_{!}j^{*}(M)\stackrel{\varepsilon_{M}} \mto M\eto  i_{*}i^{*}(M)$ in $\C$, whence we obtain the following exact sequence
$$\Hom_{\C}(j_{!}j^{*}(M),i_{*}(N))\to \Ext_{\C}^{1}(i_{*}i^{*}(M),i_{*}(N)) \to \Ext_{\C}^{1}(M,i_{*}(N)).$$
Note that $\Hom_{\C}(j_{!}j^{*}(M),i_{*}(N))\cong\Hom_{\B}(j^{*}(M),j^{*}i_{*}(N))=0$ and $\Ext_{\C}^{1}(M,i_{*}(N))=0$. It follows that $\Ext_{\C}^{1}(i_{*}i^{*}(M),i_{*}(N))=0$, and therefore $\Ext_{\A}^{1}(i^{*}(M),N)\cong\Ext_{\A}^{1}(i^{*}(M),i^{!}i_{*}(N))\cong\Ext_{\C}^{1}(M,i_{*}(N))=0$ by Lemma \ref{lem:2.11}(1). This implies $i^{*}(M)\in{\mathcal{U}}'$.

Next we assume that $L$ is an object in $\mathcal{V}''$. Then $j_{*}(L)$ is in $\mathcal{V}$. It follows from Proposition \ref{prop:2.5}(1) and Lemma \ref{lem:2.2'}(1) that $ i_{*}i^{!}(M) \mto M \stackrel{\eta_{M}} \eto j_{*}j^{*}(M)$ is an exact sequence in $\C$. Thus we have the following exact sequence
$$\Hom_{\C}(i_{*}i^{!}(M),j_{*}(L))\to \Ext_{\C}^{1}(j_{*}j^{*}(M),j_{*}(L)) \to \Ext_{\C}^{1}(M,j_{*}(L)).$$
Since $\Hom_{\C}(i_{*}i^{!}(M),j_{*}(L))\cong\Hom_{\mathcal{A}}(i^{!}(M),i^{!}j_{*}(L))=0$ and $\Ext_{\C}^{1}(M,j_{*}(L))=0$, it follows that $\Ext_{\C}^{1}(j_{*}j^{*}(M),j_{*}(L))=0$. Let $\xi:L\mto D\eto j^{*}(M)$ be an admissible exact sequence in $\B$. Since $i^{!}:\C\to\A$ is exact by hypothesis, it follows from Lemma \ref{lem:property-of-recollement1}(5) that $j_{*}:\B\to\C$ is also exact, and therefore $j_{*}(\xi): j_{*}(L)\mto j_{*}(D)\eto j_{*}j^{*}(M)$ is an admissible exact sequence in $\C$. This implies that $j_{*}(\xi)$ is split. Since $j_{*}$ is a fully faithful functor, the sequence $\xi$ is split. It follows that $\Ext_{\C}^{1}(j^{*}(M),L)=0$. Thus $j^{*}(M)\in{\mathcal{U}}''$, as desied.

``$\Leftarrow$".  Assume that $(\mathcal{U},\mathcal{V})$ is a hereditary cotorsion pair in $\C$. It is easy to check that $\mathcal{U}'=i^{\ast}(\mathcal{U})$, $\mathcal{U}''=j^{\ast}(\mathcal{U})$, $\mathcal{V}'=i^{!}(\mathcal{V})$ and $\mathcal{V}''=j^{\ast}(\mathcal{V})$. By Lemma \ref{lem:2.11}(1), we obtain $\Ext_{\mathcal{A}}^{i}(i^{\ast}(U),i^{!}(V))\cong\Ext_{\mathcal{C}}^{i}(i_{\ast}i^{\ast}(U),V)$ for every $U\in{\mathcal{U}}$, $V\in{\mathcal{V}}$ and $i=1,2$. Since $U\in{\mathcal{U}}$, it follows that $i_{*}i^{*}(U)\in{\mathcal{U}}$. This means $\Ext_{\mathcal{A}}^{i}(i^{\ast}(U),i^{!}(V))=0$ for $i=1,2$, and therefore $\Ext_{\mathcal{A}}^{i}(U',V')=0$ for every $U'\in{\mathcal{U'}}$, $V'\in{\mathcal{V'}}$ and $i=1,2$. Similarly, one can prove $\Ext_{\mathcal{B}}^{i}(U'',V'')=0$ for every $U''\in{\mathcal{U''}}$, $V''\in{\mathcal{V''}}$ and $i=1,2$.

Next we will prove that  $(\mathcal{U}',\mathcal{V}')$ is a cotorsion pair in $\mathcal{A}$. It suffices to show $(\mathcal{U}')^{\perp}\subseteq \mathcal{V}'$ and $^{\perp}(\mathcal{V}')\subseteq \mathcal{U}'$. Let $N$ be an object in $(\mathcal{U}')^{\perp}$. For any $U\in{\U}$, one has $\Ext_{\C}^{1}(U,i_{*}(N))\cong \Ext_{\A}^{1}(i^{*}(U),N)$ by Lemma \ref{lem:2.11}(2). Since $i^{*}(U)\in{\mathcal{U}'}$, we obtain
$\Ext_{\A}^{1}(i^{*}(U),N)=0$. Thus $\Ext_{\C}^{1}(U,i_{*}(N))=0$, whence $i_{*}(N)\in{\mathcal{V}}$. So $N\cong i^{!}i_{*}(N)$ belongs to $\mathcal{V}'$ and $(\mathcal{U}')^{\perp}\subseteq \mathcal{V}'$. On the other hand, for the containment $^{\perp}(\mathcal{V}')\subseteq \mathcal{U}'$, we assume that $M$ is an object in $^{\perp}(\mathcal{V}')$. For any $V\in{\mathcal{V}}$, one has $\Ext_{\C}^{1}(i_{*}(M),V)\cong \Ext_{\A}^{1}(M,i^{!}(V))$ by Lemma \ref{lem:2.11}(1). Since $i^{!}(V)\in{\mathcal{V}'}$, we obtain $\Ext_{\A}^{1}(M,i^{!}(V))=0$. It follows that $\Ext_{\C}^{1}(i_{*}(M),V)=0$. Thus $i_{*}(M)\in{\mathcal{U}}$. This implies that $M\cong i^{*}i_{*}(M)$ belongs to $\mathcal{U}'$, as desired.

Finally, to prove that  $(\mathcal{U}'',\mathcal{V}'')$ is a cotorsion pair in $\mathcal{B}$, it suffices to show $(\mathcal{U}'')^{\perp}\subseteq \mathcal{V}''$ and $^{\perp}(\mathcal{V}'')\subseteq \mathcal{U}''$.
Let $Y$ be an object in $(\mathcal{U}'')^{\perp}$. For any $U\in{\U}$, one has $\Ext_{\C}^{i}(U,j_{*}(Y))\cong \Ext_{\B}^{i}(j^{*}(U),Y)$ by Lemma \ref{lem:2.11}. Since $j^{*}(U)\in{\mathcal{U}''}$, we obtain
$\Ext_{\B}^{1}(j^{*}(U),Y)=0$. Thus $\Ext_{\C}^{1}(U,j_{*}(Y))=0$, whence $j_{*}(Y)\in{\mathcal{V}}$. So $Y\cong j^{*}j_{*}(Y)$ belongs to $\mathcal{V}''$ and $(\mathcal{U}'')^{\perp}\subseteq \mathcal{V}''$. On the other hand, for the containment $^{\perp}(\mathcal{V}'')\subseteq \mathcal{U}''$, we assume that $X$ is an object in $^{\perp}(\mathcal{V}'')$. For any $V\in{\mathcal{V}}$, one has $\Ext_{\C}^{1}(j_{!}(X),V)\cong \Ext_{\B}^{1}(X,j^{*}(V))$ by Lemma \ref{lem:2.11}(3). Since $j^{*}(V)\in{\mathcal{V}''}$, we have $\Ext_{\B}^{1}(X,j^{*}(V))=0$. It follows that $\Ext_{\C}^{1}(j_{!}(X),V)=0$, and therefore $j_{!}(X)\in{\mathcal{U}}$. So $X\cong j^{*}j_{!}(X)$ belongs to $\mathcal{U}''$, as desired. \hfill$\Box$

\vspace{1mm}
Now we can prove Theorem \ref{thm:1.1} in the introduction.

{\bf Proof of Theorem \ref{thm:1.1}.} Let $C$ be an object in $\C$. Then there exists an admissible exact sequence $j^{*}(C)\s{f_1}\mto V''\s{g_1}\eto U''$ in $\B$ with $V''\in{\mathcal{V}''}$ and $U''\in{\mathcal{U}''}$. Since $j_{!}$ is $\mathcal{U}''$-exact, we have the following pushout diagram in $\C$
$$\xymatrix{
j_{!}j^{*}(C) \ar@{>->}[r]^{j_{!}(f_1)}\ar[d]^{\varepsilon_{C}}&j_{!}(V'') \ar[d]^{\beta_1}\ar@{->>}[r]^{j_{!}(g_1)}&j_{!}(U'')\ar@{=}[d]\\
 C \ar@{>->}[r]^{\alpha_{1}}&D \ar@{->>}[r]^{\alpha_2}&j_{!}(U'').\\}$$
Since $i^{!}$ is an exact functor by hypothesis, we have an admissible exact sequence
$i^{!}(C)\s{i^{!}(\alpha_1)}\mto i^{!}(D)\s{i^{!}(\alpha_2)}\eto i^{!}j_{!}(U'')$ in $\A$.
Note that there exists an admissible exact sequence $i^{!}(D)\s{f_2}\mto V'\s{g_2}\eto U'$ in $\A$ with $V'\in{\mathcal{V}'}$ and $U'\in{\mathcal{U}'}$. Thus we have the following pushout diagram in $\A$
$$\xymatrix{
 i^{!}(C)\ar@{>->}[r]^{i^{!}(\alpha_1)}\ar@{=}[d]&i^{!}(D)\ar@{->>}[r]^{i^{!}(\alpha_2)}\ar@{>->}[d]^{f_2}&i^{!}j_{!}(U'')
\ar@{>->}[d]^{\gamma_1}\\
i^{!}(C)\ar@{>->}[r]^{h_1}&V'\ar@{->>}[d]^{g_2}\ar@{->>}[r]^{h_2}&X
\ar@{->>}[d]^{\gamma_2}\\
&U'\ar@{=}[r]&U'.\\}$$
By Lemma \ref{lem:2.6}, there is an admissible exact sequence in $\C$ $$\xymatrix@C=2.0cm{ i_{*}i^{!}j_{!}j^{*}(C) \ar@{>->}[r]^{\varphi_1}& j_{!}j^{*}(C)\oplus  i_{*}i^{!}(C) \ar@{->>}[r]^{\psi_1}&C,}$$
where $\varphi_{1}={\left(\begin{smallmatrix}\sigma_{j_{!}j^{*}(C)}\\i_{*}i^{!}(\varepsilon_{C})\end{smallmatrix}\right)}$ and $\psi_{1}=(-\varepsilon_{C},\sigma_{C})$.

Note that $\sigma_{j_{!}(V'')}:i_{*}i^{!}j_{!}(V'')\to j_{!}(V'')$ and $\sigma_{j_{!}(U'')}:i_{*}i^{!}j_{!}(U'')\to j_{!}(U'')$ are admissible monomorphisms. Then we have the following pushout diagrams
$$\xymatrix{
 i_{*}i^{!}j_{!}(V'') \ar@{>->}[r]^{\sigma_{j_{!}(V'')}}\ar[d]^{i_{*}(f_{2}i^{!}(\beta_1))}&j_{!}(V'') \ar@{.>}[d]^{a_1}\\
 i_{*}(V') \ar@{.>}[r]^{b_{1}}&V,\\}\eqno{\raisebox{-4ex}{(4.1)}}$$
$$
\xymatrix{
i_{*}i^{!}j_{!}(U'') \ar@{>->}[r]^{\sigma_{j_{!}(U'')}}\ar@{>->}[d]^{i_{*}(\gamma_1)}&j_{!}(U'') \ar@{>.>}[d]^{a_2}\\
i_{*}(X) \ar@{.>}[r]^{b_{2}}\ar@{->>}[d]^{i_{*}(\gamma_2)}&U\ar@{.>>}[d]\\
i_{*}(U') \ar@{=}[r]&i_{*}(U').\\}\eqno{\raisebox{-9ex}{(4.2)}}$$
Thus there exist admissible exact sequences in $\C$
$$\xymatrix@C=2.0cm{i_{*}i^{!}j_{!}(V'') \ar@{>->}[r]^{\varphi_2}&
j_{!}(V'')\oplus i_{*}(V')\ar@{->>}[r]^{\psi_2}&V,}$$
$$\xymatrix@C=2.0cm{i_{*}i^{!}j_{!}(U'') \ar@{>->}[r]^{\varphi_3}&
j_{!}(U'')\oplus i_{*}(X) \ar@{->>}[r]^{\psi_3}&U,}$$
where $\varphi_{2}={\left(\begin{smallmatrix}\sigma_{j_{!}(V'')}\\i_{*}(f_2)i_{*}i^{!}(\beta_1)\end{smallmatrix}\right)}$, $\varphi_{3}={\left(\begin{smallmatrix}\sigma_{j_{!}(U'')}\\i_{*}(\gamma_1)\end{smallmatrix}\right)}$, $\psi_{2}=(-a_{1},b_{1})$ and $\psi_{3}=(-a_{2},b_{2})$.
Applying $j^{*}$ to the commutative diagrams (4.1) and (4.2), we conclude that $j^{*}(a_1):j^{*}j_{!}(V'')\to j^{*}(V)$ and $j^{*}(a_2):j^{*}j_{!}(U'')\to j^{*}(U)$ are isomorphisms. Since $i_{*}(\gamma_{1})$ is admissible monic, so is $a_2$. Hence we have the following diagram
$$\xymatrix@C=2.0cm@R=1.2cm{
 i^{!}j_{!}j^{*}j_{!}(U'') \ar[r]^{i^{!}j_{!}j^{*}(a_2)}\ar[d]^{i^{!}(\varepsilon_{j_{!}(U'')})}&i^{!}j_{!}j^{*}(U) \ar[d]^{i^{!}(\varepsilon_{U})}\\
  i^{!}j_{!}(U'') \ar[r]^{i^{!}(a_2)}&i^{!}(U),\\}$$
such that ${i^{!}(\varepsilon_{j_{!}(U'')})}$ and ${i^{!}j_{!}j^{*}(a_2)}$ are isomorphisms. Since $i^{!}(a_2)$ is an admissible monomorphism, so is ${i^{!}(\varepsilon_{U})}$. Note that there exists an object $A\in{\A}$ such that $i^{*}(A) \mto j_{!}j^{*}(U)\stackrel{\varepsilon_{U}} \to U$ is a left exact sequence in $\C$. Thus $i^{!}i^{*}(A) \mto i^{!}j_{!}j^{*}(U)\stackrel{i^{!}(\varepsilon_{U})} \to i^{!}(U)$ is also a left exact sequence in $\A$, whence $A\cong i^{!}i^{*}(A)=0$. So $\varepsilon_{U}: j_{!}j^{*}(U)\to U$ is an admissible monomorphsm.
Applying the snake lemma (see \cite[Corollary 8.13]{TBh10}), we have the following commutative diagram such that all rows and columns are admissible exact sequences in $\C$
$$\xymatrix@C=1.7cm@R=1.5cm{
i_{*}i^{!}j_{!}j^{*}(C) \ar@{>->}[r]^{\varphi_1}\ar@{>->}[d]^{i_{*}i^{!}j_{!}(f_1)}&j_{!}j^{*}(C)\oplus i_{*}i^{!}(C)\ar@{>->}[d]^{\left(\begin{smallmatrix}j^{!}(f_1)&0\\0&i_{*}(h_1)\\\end{smallmatrix}\right)}\ar@{->>}[r]^{\psi_{1}}&C\ar@{>.>}[d]^{\tau_1}\\
i_{*}i^{!}j_{!}(V'') \ar@{>->}[r]^{\varphi_2}\ar@{->>}[d]^{i_{*}i^{!}j_{!}(g_1)}&j_{!}(V'')\oplus i_{*}(V')\ar@{->>}[d]^{{\left(\begin{smallmatrix}j^{!}(g_1)&0\\0&i_{*}(h_2)\\\end{smallmatrix}\right)}}\ar@{->>}[r]^{\psi_2}&V\ar@{.>>}[d]^{\tau_2}\\
i_{*}i^{!}j_{!}(U'') \ar@{>->}[r]^{\varphi_3}& j_{!}(U'')\oplus i_{*}(X) \ar@{->>}[r]^{\psi_3}&U\\}\eqno{\raisebox{-13.5ex}{(4.3)}}$$

Applying $j^{*}$ to the diagram (4.3) above, we obtain $j^{*}(V)\cong j^{*}j_{!}(V'')\cong V''\in{\mathcal{V}''}$ and $j^{*}(U)\cong j^{*}j_{!}(U'')\cong U''\in{\mathcal{U}''}$. If we set $\zeta:1_{\A}\to i^{!}i_{*}$ be the unit of the adjoint pair $(i_*,i^!)$ , then we obtain $i^!(\sigma_{j_{!}(V'')})\zeta_{i^{!}j_{!}(V'')}=\textrm{1}_{i^{!}j_{!}(V'')}$. This means that $i^{!}(\sigma_{j_{!}(V'')})$ is an admissible epimorphism.
Since $i^{!}(\sigma_{j_{!}(V'')})$ is an admissible monomorphism, it follows that $i^!(\sigma_{j_{!}(V'')})$ is an isomorphism. Applying $i^!$ to the commutative diagram (4.1) yields that $i^{!}(V)\cong i^{!}i_{*}(V')\cong V'\in{\mathcal{V}'}$, and therefore we obtain $V\in{\mathcal{V}}$. Next, applying $i^*$ to the second column in the diagram (4.2) leads to a right exact sequence $i^{*}j_{!}(U'')\s{i^*(a_{2})}\to i^*(U)\eto i^{*}i_{*}(U')$. Since $i^{*}j_{!}=0$, it follows that $i^*(U)\cong i^{*}i_{*}(U')\cong{U'\in{\mathcal{U}'}}$. This implies $U\in{\mathcal{U}}$. So $(\mathcal{U},\mathcal{V})$ is a complete cotorsion pair by Lemma \ref{lem:salce-lemma}, as desired.
\hfill$\Box$


\vspace{2mm}
By \cite[Example 2.12]{Psaroudakis}, the comma category $\mathcal{C}=(T\downarrow\mathcal{A})$ defined in introduction can induce the following recollement of abelian categories:

$$\xymatrixcolsep{5pc}\xymatrix{\ \mathcal{A}\ar[r]^{i_*}&\ \mathcal{C}\ar@/^1pc/[l]^{i^!}\ar@/_1pc/[l]_{i^*}\ar[r]^{j^*} &\ \mathcal{B}\ar@/^1pc/[l]^{j_*}\ar@/_1pc/[l]_{j_!},}\eqno{(4.4)}$$
where $i^{*}(\left(\begin{smallmatrix}A\\B\end{smallmatrix}\right)_{f})=\textrm{coker}f$, $i^{!}(\left(\begin{smallmatrix}A\\B\end{smallmatrix}\right)_{f})=A$ and $j^{*}(\left(\begin{smallmatrix}A\\B\end{smallmatrix}\right)_{f})=B$ for any $\left(\begin{smallmatrix}A\\B\end{smallmatrix}\right)_{f}\in{(T\downarrow\mathcal{A})}$,
$i^{*}(A)=\left(\begin{smallmatrix}A\\0\end{smallmatrix}\right)_{0}$ for any $A\in{\mathcal{A}}$, and $j^{!}(B)=\left(\begin{smallmatrix}T(B)\\B\end{smallmatrix}\right)_{\textrm{id}}$ and  $j_{*}(B)=\left(\begin{smallmatrix}0\\B\end{smallmatrix}\right)_{0}$ for any $B\in{\mathcal{B}}$.

\begin{rem}\label{remark:3.9} Assume that $T:\mathcal{B}\rightarrow\mathcal{A}$ is a right exact functor between abelian categories with enough projective and injective objects. Note that $i^{!}$ is an exact functor in the above recollement (4.4). It follows from Proposition \ref{prop:2.5}(1) that $\mathcal{C}_{1}=\{C\in {\mathcal{C}} \ | \ i^{*}j_{*} j^{*}(C)=0\}=\mathcal{C}$ and $\mathcal{B}_1=\{B\in{\B} \ | \ B\cong j^{*}(C)  \ \textrm{for} \ \textrm{some} \ C\in{\mathcal{C}_1}\}=\mathcal{B}$. So Theorem \ref{thm:1.1} here is just Lemma 3.3 and Proposition 3.4 in \cite{Huzhu}. Note that, in \cite{Huzhu}, one of the key arguments in the proof is that all objects in $(T\downarrow\mathcal{A})$ can be represented clearly by
the objects in $\mathcal{A}$ and $\mathcal{B}$, while in our general context we do not have this fact and therefore must avoid this kind
of arguments. So, the idea of proving Theorem \ref{thm:1.1} will be different from the one in \cite{Huzhu}.
\end{rem}

We end this section with the following example which illustrates  Theorem \ref{thm:1.1}.

\begin{ex}\label{ex:3.2} Let $A=B$ be the path algebra $k(1 \rightarrow 2)$, where $\mathrm{char} k \neq 2$. Take $M = N =
Ae_{2} \otimes_{k} e_{1}A$. The Auslander-Reiten quiver $\Gamma(\mathrm{mod}\text{-}A)$ of the module category $\text{mod-}A$ has the form
$$\xymatrix@M=1pt@!0{
   &   A{e_{1}} \ar[dr]^{\pi}\\
   S_{2}\ar[ur]^{\sigma}&& S_{1}. }$$
Keep the notation of Example \ref{ex:3.1}. Thus we have
$\mathcal{B}_{1}=\{Y\in\text{ Mod-}A\mid \mathrm{Hom}_{A}(N,Y)=0\}=\mathrm{Add}(S_{1})$, which implies
\begin{align*}
\mathcal{C}_{1}&=\{\left(\begin{smallmatrix}  X  \\   Y \\\end{smallmatrix}\right)_{f,g}\in\text{Mod-}\Lambda\mid Y\in{\B}_{1}\}\\
 &= \mathrm{Add}\{\left(\begin{smallmatrix}  S_{2}  \\   S_{1} \\\end{smallmatrix}\right)_{0,1},~\left(\begin{smallmatrix}  Ae_{1}  \\   S_{1} \\\end{smallmatrix}\right)_{0,\sigma},~\left(\begin{smallmatrix}  S_{1}  \\   S_{1} \\\end{smallmatrix}\right)_{0,0},~\left(\begin{smallmatrix}  Ae_{1}  \\  0 \\\end{smallmatrix}\right)_{0,0},~\left(\begin{smallmatrix}  S_{2}  \\  0 \\\end{smallmatrix}\right)_{0,0},~\left(\begin{smallmatrix}  S_{1}  \\   0 \\\end{smallmatrix}\right)_{0,0}\}.
\end{align*}
Thus $\mathrm{Add}\{\left(\begin{smallmatrix}  S_{2}  \\   S_{1} \\\end{smallmatrix}\right)_{0,1},~\left(\begin{smallmatrix}  Ae_{1}  \\   0 \\\end{smallmatrix}\right)_{0,0},~\left(\begin{smallmatrix}  S_{2}  \\   0 \\\end{smallmatrix}\right)_{0,0}\}$ is the class of projective objects of $\mathcal{C}_{1}$. It is an easy exercise to show that $\C_{1}$ has enough projective objects. Denote by $\textrm{Proj}(\C_{1})$ the class of projective objects of $\C_{1}$. Similarly, one can show that $\C_{1}$ has enough injective objects.

Take $(\mathcal{U}',\mathcal{V}')=(\textrm{Proj}(\C_{1}),\C)$ and $(\mathcal{U}'',\mathcal{V}'')=(\B_{1},\B_{1})$. Clearly, $j_{!}$ is $\mathcal{U}''$-exact since $\mathcal{U}''$ is the class of projective objective of $\B_{1}$. If we set $\mathcal{V}=\mathcal{C}_{1}$ and
  \begin{align*}
   \mathcal{U}&=\left\{\left(\begin{smallmatrix}  X  \\   Y \\\end{smallmatrix}\right)_{f,g}\in\text{Mod-}\Lambda\bigg|
   {\left.
\begin{aligned}\mathrm{coker}g\in \mathrm{Proj}~A,~Y\in{\B_{1}},~g\text{~is~~ monomorphic}\end{aligned}
\right.}\right\},
 \end{align*}
then $\mathcal{(U,V)}={(\mathscr{M}^{\mathcal{U}'}_{\mathcal{U}''}\,\mathscr{N}_{\mathcal{V}''}^{\mathcal{V}'})}$ is a projective cotorsion pair in $\mathcal{C}_{1}$ by Theorem \ref{thm:1.1}.

Finally, we consider the recollement (3.4) of module categories in Example \ref{ex:3.1}. It is cleat that $i^{!}:\text{Mod-} \Lambda\to \text{Mod-} A$ is not an exact functor.
Moreover, if we set $(\mathcal{U}',\mathcal{V}')=(\mathcal{U}'',\mathcal{V}'')=(\mathrm{Add}\{S_{2},Ae_{1}\},\text{Mod-} A)$ in , then
$$\mathscr{M}^{\mathcal{U}'}_{\mathcal{U}''}=\{\left(\begin{smallmatrix}  X  \\   Y \\\end{smallmatrix}\right)_{f,g}\in\text{Mod-}\Lambda\mid
\mathrm{coker}g\in \mathrm{Proj}~A,~Y\in \mathrm{Proj}~A,~g \text{~is~~ monomorphic}
  \}$$
 and $\mathscr{N}_{\mathcal{V}''}^{\mathcal{V}'}=\text{Mod-} \Lambda$. If follows from \cite[Theorem 4.4]{ZH} that $(\mathscr{M}^{\mathcal{U}'}_{\mathcal{U}''},\mathscr{N}_{\mathcal{V}''}^{\mathcal{V}'})$ is no longer a cotorsion pair in $\text{Mod-} \Lambda$. So the exactness of the functor $i^!$ in Theorem \ref{thm:1.1} cannot be omitted in general.
\end{ex}

\subsection{Proof of Theorem \ref{thm:main1}}\label{Models and recollement}

We begin this subsection with the following lemma, which provides us a method to construct a Hovey triple from two cotorsion pairs in a WIC exact category.

\begin{lem} {\rm \cite[Theorem 1.1]{gj14}\label{lem4.2}} Let $\mathcal{C}$ be a WIC exact category and suppose $(\mathcal{Q},\widetilde{\mathcal{R}})$ and $(\widetilde{\mathcal{Q}},\mathcal{R})$ are complete hereditary cotorsion pairs over $\mathcal{C}$ with (1) $\widetilde{\mathcal{Q}}\subseteq \mathcal{Q}$, (2) $\mathcal{Q}\cap\widetilde{\mathcal{R}}=\widetilde{\mathcal{Q}}\cap\mathcal{R}$. Then there exists a unique exact model
structure $(\mathcal{Q}, \mathcal{W} , \mathcal{R})$, and its class $\mathcal{W}$ of trivial objects is given by
\begin{align*}
 \mathcal{W }&=\{ X \in\mathcal{C} \mid \exists~\text{an admissible exact sequence}~ X\mto R\eto Q~with~R\in\widetilde{\mathcal{R}},~Q\in\widetilde{\mathcal{Q}} \}\\
&=\{ X \in\mathcal{C} \mid \exists~\text{an admissible exact sequence}~R'\mto Q'\eto X~with~R'\in\widetilde{\mathcal{R}},~Q'\in\widetilde{\mathcal{Q}}\}.
\end{align*}
\end{lem}

 Let $\mathcal{M}_{\mathcal{A}}=(\mathcal{U}'_{1},\mathcal{W}',\mathcal{V}'_{2})$ be a hereditary exact model structure on $\mathcal{A}$. Then we have two complete hereditary cotorsion pairs $(\mathcal{U}'_{1},\mathcal{V}'_{1})$, $(\mathcal{U}'_{2},\mathcal{V}'_{2})$ in $\mathcal{A}$, where $\mathcal{V}'_{1}=\mathcal{W}'\cap\mathcal{V}'_{2}$, $\mathcal{U}'_{2}=\mathcal{U}'_{1}\cap\mathcal{W}'$.   Let $\mathcal{M}_{\mathcal{B}}=(\mathcal{U}''_{1},\mathcal{W}'',\mathcal{V}''_{2})$ be an exact model structure on $\mathcal{B}$. Similarly, we obtain two complete hereditary cotorsion pairs $(\mathcal{U}''_{1},\mathcal{V}''_{1})$, $(\mathcal{U}''_{2},\mathcal{V}''_{2})$ in $\mathcal{B}$, where $\mathcal{V}''_{1}=\mathcal{W}''\cap\mathcal{V}''_{2}$, $\mathcal{U}''_{2}=\mathcal{U}''_{1}\cap\mathcal{W}''$.
 Therefore, if $j_{!}$ is $\mathcal{U}''_{1}$-exact, from  Theorem \ref{thm:1.1} we will obtain two complete hereditary cotorsion pairs  $(\mathcal{U}_{1},\mathcal{V}_{1})$  and $(\mathcal{U}_{2},\mathcal{V}_{2})$  in $\C$, where $\mathcal{U}_1=\mathscr{M}^{\mathcal{U}'_{1}}_{\mathcal{U}''_{1}}$, $\mathcal{V}_1=\mathscr{N}_{\mathcal{V}''_{1}}^{\mathcal{V}'_{1}}$, $\mathcal{U}_2=\mathscr{M}^{\mathcal{U}'_{2}}_{\mathcal{U}''_{2}}$ and $\mathcal{V}_2=\mathscr{N}_{\mathcal{V}''_{2}}^{\mathcal{V}'_{2}}$.
  If $\mathcal{U}_{1}\cap\mathcal{V}_{1}=\mathcal{U}_{2}\cap\mathcal{V}_{2}$, then by Lemma \ref{lem4.2}, there exists a unique class $\mathcal{W}$, such that  $\mathcal{M}_{\mathcal{C}}=(\mathcal{U}_{1},\mathcal{W},\mathcal{V}_{2})$ is a Hovey triple in $\mathcal{C}$, and $\mathcal{W}\cap\mathcal{V}_{2}=\mathcal{V}_{1}$, $\mathcal{U}_{1}\cap\mathcal{W}=\mathcal{U}_{2}$.

A \emph{Quillen map} of model categories $\mathcal{M} \rightarrow \mathcal{N}$ consists of a pair
of adjoint functors $(L, R) : \mathcal{M} \rightleftarrows \mathcal{N}$ such that $L$ preserves cofibrations and trivial
cofibrations (it is equivalent to require that $R$ preserves fibrations and trivial fibrations). In this case the pair $(L,R)$ is also called a \emph{Quillen adjunction}. A Quillen map induces adjoint total derived functors between the homotopy
categories \cite{ma12}. The class of weak equivalences is the most important class of morphisms in a model category. The following important characterization
is proved in \cite[Corollary 3.4]{gj11}.

\begin{lem}\label{lem4.3} Let $\mathcal{M= (Q,W,R)}$ be a Hovey triple. Then a morphism is a weak
equivalence if and only if $f=pi$, where $i$ is an admissible monomorphism with ${\rm coker}i\in{\mathcal{Q}\cap\mathcal{W}}$ and $p$ is an admissible epimorphism with ${\rm ker}p\in{\mathcal{R}\cap\mathcal{W}}$.
\end{lem}

Before giving our main result, we need the following crucial result.

\begin{prop}\label{pro4.4} Let $(\mathcal{A},\mathcal{C},\mathcal{B})$ be a recollement of exact categories with $i^{!}$ exact, and let $\mathcal{M}_{\mathcal{A}}=(\mathcal{U}'_{1},\mathcal{W}',\mathcal{V}'_{2})$ and $\mathcal{M}_{\mathcal{B}}=(\mathcal{U}''_{1},\mathcal{W}'',\mathcal{V}''_{2})$ be hereditary exact model structures on $\mathcal{A}$ and $\mathcal{B}$, respectively. We set $\mathcal{U}_1=\mathscr{M}^{\mathcal{U}'_{1}}_{\mathcal{U}''_{1}}$, $\mathcal{V}_1=\mathscr{N}_{\mathcal{W}''\cap\mathcal{V}''_{2}}^{\mathcal{W}'\cap\mathcal{V}'_{2}}$, $\mathcal{U}_2=\mathscr{M}^{\mathcal{U}'_{1}\cap\mathcal{W}'}_{\mathcal{U}''_{1}\cap\mathcal{W}''}$ and $\mathcal{V}_2=\mathscr{N}_{\mathcal{V}''_{2}}^{\mathcal{V}'_{2}}$. Assume that $j_{!}$ is $\mathcal{U}''_{1}$-exact and  $\mathcal{U}_{1}\cap\mathcal{V}_{1}=\mathcal{U}_{2}\cap\mathcal{V}_{2}$. Then the following hold.

\begin{enumerate}
\item There is a hereditary exact model structure $\mathcal{M}_{\mathcal{C}}=(\mathcal{U}_{1},\mathcal{W},\mathcal{V}_{2})$ on $\mathcal{C}$, where the class $\mathcal{W}$ is given by
\begin{align*}
 \mathcal{W }&=\{ X \in\mathcal{C} \mid \exists~\text{an admissible exact sequence}~  X\mto R\eto Q~with~R\in{\mathcal{V}_1},~Q\in{\mathcal{U}_2} \}\\
&=\{ X \in\mathcal{C} \mid \exists~\text{an admissible exact sequence}~  R'\mto Q'\eto X~with~R'\in{\mathcal{V}_1},~Q'\in{\mathcal{U}_2}\}.
\end{align*}

\item We have the following localization sequence of triangulated categories
$$\xymatrixcolsep{5pc}\xymatrix{
  \mathrm{Ho}(\mathcal{M}_{\mathcal{A}}) \ar@<0.6ex>[r]^{L(i_{\ast})} & \mathrm{Ho}(\mathcal{M}_{\mathcal{C}})\ar@/^2pc/[l]^{R(i^{!})}\ar@<0.6ex>[r]^{L(j^{\ast})} & \mathrm{Ho}(\mathcal{M}_{\mathcal{B}})\ar@/^2pc/[l]^{R(j_{\ast})} }$$
 where $L(i_{\ast})$, $L(j^{\ast})$, $R(i^{!})$ and $R(j_{\ast})$ are the total derived functors of those in (1.1).

 \item Then we have the following colocalization sequence of triangulated categories
$$\xymatrixcolsep{5pc}\xymatrix{
  \mathrm{Ho}(\mathcal{M}_{\mathcal{A}}) \ar@<-0.6ex>[r]_-{R(i_{\ast})} & \mathrm{Ho}(\mathcal{M}_{\mathcal{C}})\ar@/_2pc/[l]_{L(i^{\ast})}\ar@<-0.6ex>[r]_{R(j^{\ast})} & \mathrm{Ho}(\mathcal{M}_{\mathcal{B}})\ar@/_2pc/[l]_{L(j_{!})} }$$
where $L(i^{\ast})$, $L(j_{!})$, $R(i_{\ast})$ and $R(j^{\ast})$ are the total derived functors of those in (1.1).
\end{enumerate}

\end{prop}
\begin{proof}
(1) Let $\mathcal{V}'_{1}=\mathcal{W}'\cap\mathcal{V}'_{2}$, $\mathcal{U}'_{2}=\mathcal{U}'_{1}\cap\mathcal{W}'$ and $\mathcal{V}''_{1}=\mathcal{W}''\cap\mathcal{V}''_{2}$, $\mathcal{U}''_{2}=\mathcal{U}''_{1}\cap\mathcal{W}''$.
If $j_{!}$ is $\mathcal{U}''_{1}$-exact,  then by Theorem \ref{thm:1.1}, $(\mathcal{U}_{1},\mathcal{V}_{1})$ and  $(\mathcal{U}_{2},\mathcal{V}_{2})$ are two complete cotorsion pairs in $\mathcal{C}$. Because $\mathcal{U}_{2}=\mathfrak{M}^{\mathcal{U}'_{1}\cap\mathcal{W}'}_{\mathcal{U}''_{1}\cap\mathcal{W}''}\subseteq \mathfrak{M}^{\mathcal{U}'_{1}}_{\mathcal{U}''_{1}}=\mathcal{U}_{1}$ and  $\mathcal{U}_{1}\cap\mathcal{V}_{1}=\mathcal{U}_{2}\cap\mathcal{V}_{2}$, by Lemma \ref{lem4.2}, there exists a unique class $\mathcal{W}$, such that $\mathcal{M}_{\mathcal{C}}=(\mathcal{U}_{1},\mathcal{W},\mathcal{V}_{2})$ is a hereditary exact model structure on $\mathcal{C}$.

(2) We first claim that $(i_{\ast},i^{!})$ and $(j^{\ast},j_{\ast})$ are Quillen adjunctions. Since (trivial) cofibrations equal admissible monomorphisms with (trivially) cofibrant cokernels and (trivial) fibrations equal admissible epimorphisms with (trivially) fibrant kernels, the inclusions $i_{\ast}(\mathcal{U}'_{1})\subseteq\mathcal{ U}_{1}$ and $i_{\ast}(\mathcal{U}'_{1}\cap\mathcal{W}')\subseteq {\mathcal{U}_{1}\cap\mathcal{W}}$ imply that $i_{\ast}$ preserves cofibrations and trivial cofibrations. Thus $(i_{\ast},i^{!})$ is a Quillen adjunction.
Similarly, $j^{\ast}$ is a left adjoint and preserves cofibrations and trivial cofibrations. Hence $(j^{\ast},j_{\ast})$ is a Quillen adjunction by the definition.
By \cite[Proposition 16.2.2]{ma12}, the total derived functors $L(i_{\ast})$ and $R(i^{!})$ exist and form an adjoint between $\mathrm{Ho}(\mathcal{M}_{\mathcal{A}})$ and $\mathrm{Ho}(\mathcal{M}_{\mathcal{C}})$,  $L(j^{\ast})$ and $R(j_{\ast})$ exist and form an adjoint between $\mathrm{Ho}(\mathcal{M}_{\mathcal{C}})$ and $\mathrm{Ho}(\mathcal{M}_{\mathcal{B}})$.
That is, we have the following diagram
$$\xymatrixcolsep{5pc}\xymatrix{
  \mathrm{Ho}(\mathcal{M}_{\mathcal{A}}) \ar@<0.6ex>[r]^{L(i_{\ast})} & \mathrm{Ho}(\mathcal{M}_{\mathcal{C}})\ar@/^2pc/[l]^{R(i^{!})}\ar@<0.6ex>[r]^{L(j^{\ast})} & \mathrm{Ho}(\mathcal{M}_{\mathcal{B}})\ar@/^2pc/[l]^{R(j_{\ast})}.}\eqno{(4.5)}$$
In general, the right derived functor is defined on objects by first taking a fibrant replacement and then applying the functor. Similarly, the left derived functor is defined by first taking a cofibrant replacement and then applying the functor. So we have computed $(L(i_{\ast}),R(i^{!}))=(i_{\ast}Q_{\mathcal{A}},i^{!}R_{\mathcal{C}})$ and $(L(j^{\ast}),R(j_{\ast}))=(j^{\ast}Q_{\mathcal{C}},j_{\ast}R_{\mathcal{B}})$. Here, the notation such as $Q_{\mathcal{A}}$ means to take a special $\mathcal{U}'_{1}$-precover. Similarly the notation $R_{\mathcal{C}}$ means to take a special $\mathcal{V}_{2}$-preenvelope.
Recall from \cite[Proposition 4.4 and Section 5]{gj11} that the distinguished triangles in Ho$(\mathcal{M})$
are, up to isomorphism, the images in Ho$(\mathcal{M})$ of distinguished triangles in $(\mathcal{Q}\cap \mathcal{R})/\omega$ under
the equivalence $(\mathcal{Q}\cap \mathcal{R})/\omega\rightarrow \mathrm{Ho}(\mathcal{M})$.
By an argument similar to that in \cite[Corollary 2.10]{GC}, we see that these four functors  are triangulated
functors.

In order to show that the diagram (4.5) is a localization sequence, it remains to show

(i) $R(i^{!})\circ L(i_{\ast})\cong 1_{\mathrm{Ho}(\mathcal{M}_{\mathcal{A}})}$.

(ii) $L(j^{\ast})\circ R(j_{\ast})\cong 1_{\mathrm{Ho}(\mathcal{M}_{\mathcal{B}})}$.

(iii) The essential image of $L(i_{\ast})$ equals the kernel of $L(j^{\ast})$.

To prove (i), let $f:X\rightarrow Y$ be a homomorphism in $\mathcal{A}$.
Using the completeness of the cotorsion pair $(\mathcal{U}'_{1}\cap\mathcal{W}',\mathcal{V}'_{2})$, we get the following commutative diagram
$$\xymatrix{
   X \ar[d]_{f} \ar@{>->}[r]^{q} & X' \ar[d]_{\widetilde{f}} \ar@{->>}[r]^{} & Z_{1}  \\
   Y \ar@{>->}[r]^{q'} & Y' \ar@{->>}[r]^{} & Z'_{1}, }$$
  where $X',Y'\in \mathcal{V}'_{2}$, $Z_{1},Z'_{1}\in \mathcal{U}'_{1}\cap\mathcal{W}'$. Note that $X'$ and $Y'$ are  fibrant replacements of $X$ and $Y$ in $\mathcal{M}_{\mathcal{A}}$, respectively. So both $q$ and $q'$ are  natural isomorphisms in $\mathrm{Ho}(\mathcal{M}_{\mathcal{A}})$. The functor $i_{\ast}Q_{\mathcal{A}}$ acts by $\widetilde{f}\mapsto \widehat{f}$, where $\widehat{f}$ is any map making the diagram below commute
$$\xymatrix{
   i_{\ast}K  \ar@{>->}[r]^{} & i_{\ast}H_{1}\ar[d]_{\widehat{f}} \ar@{->>}[r]^{j} & i_{\ast}X' \ar[d]_{i_{\ast}\widetilde{f}}  \\
   i_{\ast}K' \ar@{>->}[r]^{} & i_{\ast}H'_{1} \ar@{->>}[r]^{j'} & i_{\ast}Y' ,   }$$
 where the rows are admissible exact sequences, $H_{1},H'_{1}\in\mathcal{U}'_{1}$, and $K,K'\in \mathcal{W}'\cap\mathcal{V}'_{2}$. Moreover, we obtain $H_{1},H'_{1}\in\mathcal{U}'_{1}\cap\mathcal{V}'_{2}$ since $\mathcal{V}'_{2}$ is closed under extensions. Now applying $i^{!}R_{\mathcal{C}}$ to $\widehat{f}$ gives us $\overline{f}$ in the next commutative diagram
$$\xymatrix{
   i^{!}i_{\ast}H_{1} \ar[d]_{i^{!}\widehat{f}} \ar@{>->}[r]^{p} & i^{!}L_{1} \ar[d]_{\overline{f}} \ar@{->>}[r]^{} & i^{!}C_{1}  \\
   i^{!}i_{\ast}H'_{1} \ar@{>->}[r]^{p'} & i^{!}L'_{1} \ar@{->>}[r]^{} & i^{!}C'_{1},   }$$
where $L_{1}, L'_{1}\in \mathcal{V}_{2}$, $C_{1}, C'_{1}\in \mathcal{U}_{1}\cap\mathcal{W}=\mathcal{U}_{2}$.
 Since $i_{\ast}H_{1},i_{\ast}H'_{1}\in \mathcal{V}_{2}$ and $\mathcal{V}_{2}$ is coresolving, we get that $C_{1}, C'_{1}\in\mathcal{U}_{2}\cap\mathcal{V}_{2}=\mathcal{U}_{1}\cap\mathcal{V}_{1}$ by hypotheses.

Furthermore, it is easy to check the inclusions $i_{\ast}(\mathcal{W}'\cap\mathcal{V}'_{2})=i_{\ast}(\mathcal{V}'_{1})\subseteq \mathcal{V}_{1}=\mathcal{W}\cap\mathcal{V}_{2}$ and $i^{!}(\mathcal{U}_{1}\cap\mathcal{V}_{1})\subseteq i^{!}(\mathcal{V}_{1})\subseteq \mathcal{V}'_{1}=\mathcal{W}'\cap\mathcal{V}'_{2}$. Thus $i^{!}j,~i^{!}j',~p$ and $p'$ are all weak equivalences in $\mathcal{M}_{\mathcal{A}}$ by Lemma \ref{lem4.3}. So, in $\mathrm{Ho}(\mathcal{M}_{\mathcal{A}})$, we have a commutative diagram
$$\xymatrix{
 X\ar[r]^{q}\ar[d]_{f}& X' \ar[d]_{\widetilde{f}} \ar[r]^{\nu_{X'}} & i^{!}i_{\ast}X' \ar[d]_{ i^{!}i_{\ast}\widetilde{f}} & i^{!}i_{\ast}H_{1} \ar[r]^{p} \ar[l]_{i^{!}j} \ar[d]^{i_{\ast}\widehat{f}} & i^{!}L_{1}\ar[d]^{\overline{f}} \\
Y\ar[r]^{q'}&  Y' \ar[r]^{\nu_{Y'}} & i^{!}i_{\ast}Y' & i^{!}i_{\ast}H'_{1} \ar[l]_{i^{!}j'}\ar[r]^{p'} & i^{!}L_{2},  }$$
where $\nu:1_{\mathcal{A}}\rightarrow i^{!}i_{\ast}$ is the unit of the adjoint pair $(i_{\ast}, i^{!})$. This diagram gives rise to a natural isomorphism: $R(i^{!})\circ L(i_{\ast})\cong 1_{\mathrm{Ho}(\mathcal{M}_{\mathcal{A}})}.$

Next we prove (ii).  Let $X\in \mathrm{Ho}(\mathcal{M}_{\mathcal{B}})$ be any object. Using the completeness of the cotorsion pair $(\mathcal{U}''_{2},\mathcal{V}''_{2})$, we obtain an admissible exact sequence $X\rightarrowtail E\twoheadrightarrow L$ with $E\in\mathcal{V}''_{2}$ and $L\in \mathcal{U}''_{2}$. Note that $E$ is a fibrant replacement of $X$ in $\mathcal{M}_{\mathcal{B}}$, so we have a natural isomorphism $X\cong E$ in $\mathrm{Ho}(\mathcal{M}_{\mathcal{B}})$. By Lemma \ref{lem:property-of-recollement1}(5), $j_{\ast}$ is exact. Then the functor $R(j_{\ast})=j_{\ast}R_{\mathcal{B}}$ acts by $X\mapsto j_{\ast}E$, where $j_{\ast}E$ is in the admissible exact sequence $j_{\ast}X\rightarrowtail j_{\ast}E\twoheadrightarrow j_{\ast}L$.
Now applying $j^{\ast}Q_{\mathcal{C}}$ to $j_{\ast}E$ gives us $j^{\ast}N$ in the next admissible exact sequence
$$ j^{\ast}K\rightarrowtail j^{\ast}N\overset{\mu}{\twoheadrightarrow} j^{\ast}j_{\ast}E,$$
where $N$ is a cofibrant replacement of $j_{\ast}E$, $N\in\mathcal{U}_{1}$, $K\in \mathcal{W}\cap\mathcal{V}_{2}$. By  inclusion  $j^{\ast}(\mathcal{W}\cap\mathcal{V}_{2})=j^{\ast}(\mathcal{V}_{1})\subseteq \mathcal{V}''_{1}=\mathcal{W}''\cap\mathcal{V}''_{2}$, we see that $\mu$ is a weak equivalence in $\mathcal{M}_{\mathcal{B}}$. Hence, we have isomorphisms $L(j^{\ast})\circ R(j_{\ast})(X)\cong j^{\ast}N\stackrel{\mu}\cong j^{\ast}j_{\ast}E\cong E\cong X$ in $\mathrm{Ho}(\mathcal{M}_{\mathcal{B}})$. By an argument similar to that in (i), we see that these isomorphisms are natural.

For (iii),  let $X$ belongs to the essential image of $L(i_{\ast})$.  Then there exist an object $Y\in\mathcal{U}'_{1}$, such that $X\cong i_{\ast}Y$ in $\mathrm{Ho}(\mathcal{M}_{\mathcal{C}})$. Since $i_{\ast}Y\in \mathcal{U}_{1}$,  $L(j^{\ast})(X)=j^{\ast}i_{\ast}X=0$. Hence $X$ is contained in  kernel of $L(j^{\ast})$.

Conversely, let $X$ belongs to the  kernel of $L(j^{\ast})$. Then $L(j^{\ast})(X)$ is  a zero object in Ho$(\mathcal{M}_{\mathcal{B}})$, that is, $L(j^{\ast})(X) \in \mathcal{U}''_{1}\cap \mathcal{V}''_{2}\cap \mathcal{W}''$. We claim that there exists $Y\in \mathcal{A}$  such that $L(i_{\ast})(Y)\cong X$ in $\mathrm{Ho}(\mathcal{M}_{\mathcal{C}})$.
Notice that the functor $L(j^{\ast})$ acts by $X\mapsto j^{\ast}P$, where $j^{\ast}P$ is in an admissible exact sequence $ j^{\ast}K\rightarrowtail j^{\ast}P\twoheadrightarrow j^{\ast}X$ in $\mathcal{C}$. Here $P\in\mathcal{U}_{1}$, $K\in \mathcal{V}_{1}=\mathcal{W}\cap \mathcal{V}_{2}$.
So $j^{\ast}P\cong L(j^{\ast})(X) \in \mathcal{U}''_{1}\cap \mathcal{V}''_{2}\cap \mathcal{W}''\subseteq \mathcal{U}''_{2}$.
Now consider the admissible exact sequence
$$ j_{!}j^{\ast}P\rightarrowtail P\overset{\rho}{\twoheadrightarrow} i_{\ast}i^{\ast}P.$$
One has $j_{!}j^{\ast}P\in  j_{!}(\mathcal{U}''_{2}) \subseteq\mathcal{U}_{2}=\mathcal{U}_{1}\cap \mathcal{W}$.  It follows that $\rho$ is a weak equivalence in $\mathcal{M}_{\mathcal{C}}$.
Define $Y:= i^{\ast}(P)$. Since $i^{\ast}(P)\in i^{\ast}(\mathcal{U}_{1})\subseteq \mathcal{U}'_{1}$, we have $L(i_{\ast})(Y)=L(i_{\ast})(i^{\ast}(P))=i_{\ast}i^{\ast}P\overset{\rho^{-1}}{\cong} P\cong X$ in Ho($\mathcal{M}_{\mathcal{C}}$). Hence the desired result follows immediately.

(3) The proof is similar to that of (2), and so we omit it here.
\end{proof}

Now, we are ready to prove the main result of this paper.

{\bf Proof of Theorem \ref{thm:main1}.} By Proposition \ref{pro4.4}, we only need to show that there are natural isomorphisms $L(i_{\ast})\cong R(i_{\ast})$ and $L(j^{\ast})\cong R(j^{\ast})$. Let $\mathcal{V}'_{1}:=\mathcal{W}'\cap\mathcal{V}'_{2}$, $\mathcal{U}'_{2}:=\mathcal{U}'_{1}\cap\mathcal{W}'$ and $\mathcal{V}''_{1}:=\mathcal{W}''\cap\mathcal{V}''_{2}$, $\mathcal{U}''_{2}:=\mathcal{U}''_{1}\cap\mathcal{W}''$. Let $f: X\rightarrow Y$ be a morphism in $\mathrm{Ho}(\mathcal{M}_{\mathcal{C}})$. The functor $L(i_{\ast})$ acts by $f\mapsto \overline{f}$, where $\overline{f}$ is any morphism making the diagram below commute
$$\xymatrix{
   i_{\ast}K_{1}  \ar@{>->}[r]^{} & i_{\ast}P_{1}\ar[d]_{\overline{f}} \ar@{->>}[r]^{j_{1}} & i_{\ast}X \ar[d]_{i_{\ast}f}   \\
   i_{\ast}K_{2} \ar@{>->}[r]^{} & i_{\ast}P_{2} \ar@{->>}[r]^{j_{2}} & i_{\ast}Y.   }$$
Here all rows are admissible exact sequences, $P_{1},P_{2}\in\mathcal{U}'_{1}$, and $K_{1},K_{2}\in \mathcal{W}'\cap \mathcal{V}'_{2}=\mathcal{V}'_{1}$. The functor $R(i_{\ast})$ acts by $f\mapsto \widehat{f}$, where $\widehat{f}$ is any morphism making the next diagram commute
$$\xymatrix{
   i_{\ast}X \ar[d]_{i_{\ast}f} \ar@{>->}[r]^{q_{1}} & i_{\ast}D_{1} \ar[d]_{\widehat{f}} \ar@{->>}[r]^{} & i_{\ast}C_{1}  \\
   i_{\ast}Y \ar@{>->}[r]^{q_{2}} & i_{\ast}D_{2} \ar@{->>}[r]^{} & i_{\ast}C_{2} ,   }$$
where $D_{1},~D_{2}\in\mathcal{V}'_{2}$, $C_{1},~C_{2}\in\mathcal{U}'_{1}\cap \mathcal{W}'=\mathcal{U}'_{2}$. Note that $i_{\ast}(\mathcal{V}'_{1})\subseteq \mathcal{V}_{1}=\mathcal{W}\cap \mathcal{V}_{2}$ and $i_{\ast}(\mathcal{U}'_{2})\subseteq \mathcal{U}_{2}=\mathcal{U}_{1}\cap \mathcal{W}$. Then $j_{1},~j_{2},~q_{1}$ and $q_{2}$ are weak equivalences in $\mathcal{M}_{\mathcal{C}}$. Hence in $\mathrm{Ho}(\mathcal{M}_{\mathcal{C}})$, we have a commutative diagram
$$\xymatrix{
   i_{\ast}P_{1} \ar[d]_{\overline{f}} \ar[r]^{j_{1}} &  i_{\ast}X \ar[d]_{} \ar[r]^{q_{1}} &  i_{\ast}D_{1} \ar[d]^{\widehat{f}} \\
   i_{\ast}P_{2} \ar[r]^{j_{2}} &  i_{\ast}Y \ar[r]^{q_{2}} &  i_{\ast}D_{2} }
$$
giving rise to a natural isomorphism $L(i_{\ast})\cong R(i_{\ast})$.
The proof of the natural isomorphism $L(j^{\ast})\cong R(j^{\ast})$ is similar. This completes the proof.
\hfill$\Box$

\section{Applications to upper triangular matrix rings}\label{application}
Throughout this section, for any ring $R$, all $R$-modules
are understood to be left $R$-modules and $\mathbf{C}(R)$ is the category of chain complexes of $R$-modules. We denote by $\mathrm{Mod}\text{-}R$ the class of $R$-modules.

Let $\Lambda=\left(\begin{smallmatrix}  R & M \\  0 & S \\\end{smallmatrix}\right)$ be an upper triangular matrix ring, where $R$ and $S$ are rings and $_{R}M_{S}$ is an $R$-$S$-bimodule.
If we set $T:= M\otimes_{S}-: \mathrm{Mod}\text{-}S\rightarrow\mathrm{Mod}\text{-}R$, then $T$ induces a functor $T:\mathbf{C}(S)\rightarrow\mathbf{C}(R)$ by $X^{\bullet}\mapsto M\otimes_{S}X^{\bullet}$. Note that in this case, $\mathbf{C}(\Lambda)=(T\downarrow \mathbf{C}(S))$, where $(T\downarrow \mathbf{C}(S))=\{\left(\begin{smallmatrix}  X^{\bullet}  \\   Y^{\bullet} \\\end{smallmatrix}\right)_{\phi}\mid X^{\bullet}\in\mathbf{C}(R),~Y^{\bullet}\in \mathbf{C}(S),~\phi:TY^{\bullet}\rightarrow X^{\bullet}~\mathrm{in}~\mathbf{C}(R)\}$. Therefore, by \cite[Example 2.12]{Psaroudakis}, we obtain the recollement
$$\xymatrix@C=50pt@R=50pt{\ \mathbf{C}(R)\ar[r]^{i_*}&\ \mathbf{C}(\Lambda)\ar@/^1pc/[l]^{i^!}\ar@/_1pc/[l]_{i^*}\ar[r]^{j^*} &\ \mathbf{C}(S),\ar@/^1pc/[l]^{j_*}\ar@/_1pc/[l]_{j_!} }\eqno{(5.1)}$$
where $i^{\ast}$ is given by $\left(\begin{smallmatrix}  X^{\bullet}  \\   Y^{\bullet} \\\end{smallmatrix}\right)_{\phi}\mapsto \mathrm{coker}\phi$; $i_{\ast}$ is given by $X^{\bullet}\mapsto \left(\begin{smallmatrix}  X^{\bullet}  \\   0 \\\end{smallmatrix}\right)$; $i^{!}$ is given by $\left(\begin{smallmatrix}  X^{\bullet}  \\   Y^{\bullet} \\\end{smallmatrix}\right)_{\phi}\mapsto X^{\bullet}$; $j_{!}$ is given by $Y^{\bullet}\mapsto \left(\begin{smallmatrix}  M\otimes _{S}Y^{\bullet}  \\   Y^{\bullet} \\\end{smallmatrix}\right)_{id}$; $j^{\ast}$ is given by $\left(\begin{smallmatrix}  X^{\bullet}  \\   Y^{\bullet} \\\end{smallmatrix}\right)_{\phi}\mapsto Y^{\bullet}$; $j_{\ast}$ is given by $Y^{\bullet}\mapsto \left(\begin{smallmatrix}  0  \\   Y^{\bullet} \\\end{smallmatrix}\right)$.
  Note that the functor $i_{\ast}$, $i^{!}$, $j^{\ast}$, $j_{\ast}$ defined above are exact.

\subsection{The category of chain complexes} For a given class $\mathcal{X}$ of $R$-modules, we have the following
classes of chain complexes in $\mathbf{C}(R)$.

(1) $\widetilde{\mathcal{X}}_{R}$ denotes the class of all exact chain complexes $X$ with cycles $Z_{n}X\in \mathcal{X}$.

(2) $dw\widetilde{\mathcal{X}}_{R}$ denotes the class of all chain complexes $X$ with components $X_{n}\in \mathcal{X}$.

(3) $ex\widetilde{\mathcal{X}}_{R}$ denotes the class of all exact chain complexes $X$ with components $X_{n}\in \mathcal{X}$.

Denote by $\mathrm{Proj}\textrm{-}R$ the class of projective modules. It follows that the projective cotorsion pair $(\mathrm{Proj}\textrm{-}R,\mathrm{Mod}\text{-}R)$ in $\mathrm{Mod}\text{-}R$ can be lifted to a complete hereditary cotorsion pair $(dw \widetilde{\mathcal{P}}_{R},\mathcal{W}_{ctr,R})$ in $\mathbf{C}(R)$. The complexes in $\mathcal{W}_{ctr,R}$ have been called \emph{contraacyclic}. From \cite[Proposition 6.5]{gj16b}, we know that the triple
$\mathcal{M}^{\mathrm{proj}}_{ctr,R} = (dw \widetilde{\mathcal{P}}_{R},\mathcal{W}_{ctr,R},\mathbf{C}(R))$ is a hereditary abelian
model structure on $\mathbf{C}(R)$ and its homotopy category, $\mathrm{Ho}(\mathcal{M}^{\mathrm{proj}}_{ctr,R})$, called as the \emph{contraderived category} over $R$, is equivalent to $\mathbf{K}(\mathrm{Proj}\textrm{-}R)$, where $\mathbf{K}(\mathrm{Proj}\textrm{-}R)$
 is the chain homotopy category of all complexes of projective modules.

Now we have the following result.

\begin{cor}\label{corollary:5.5}Let $\Lambda=\left(\begin{smallmatrix}  R & M \\  0 & S \\\end{smallmatrix}\right)$ be an upper triangular matrix ring. If $\mathrm{pd}_{R}M < \infty$ and $\mathrm{pd}M_{S} < \infty$, then we have the following recollements of triangulated categories:
$$
\xymatrix@C=50pt@R=65pt{ \mathrm{Ho}(\mathcal{M}^{\mathrm{proj}}_{ctr,R})\ar[d]^{F_{R}}\ar[r]^{i_*}&\ \mathrm{Ho}(\mathcal{M}^{\mathrm{proj}}_{ctr,\Lambda})\ar[d]^{F_{\Lambda}}\ar@/_2pc/[l]_{i^*Q_{\Lambda}}\ar@/^2pc/[l]^{i^{!}}
\ar[r]^{j^{\ast}} &\ \mathrm{Ho}(\mathcal{M}^{\mathrm{proj}}_{ctr,S})\ar[d]^{F_{S}}\ar@/_2pc/[l]_{j_!Q_{S}}\ar@/^2pc/[l]^{j_\ast}  \\
\ \mathbf{K}(\mathrm{Proj}\textrm{-}R)\ar[r]^{F_{\Lambda}i_*F_{R}^{-1}}&\ \mathbf{K}(\mathrm{Proj}\textrm{-}\Lambda)\ar@/_2pc/[l]_{F_{R}i^*Q_{\Lambda}F_{\Lambda}^{-1}}\ar@/^2pc/[l]^{F_{R}i^{!}
F^{-1}_{\Lambda}}
\ar[r]^{F_{S}j^{\ast}F_{\Lambda}^{-1}} &\ \mathbf{K}(\mathrm{Proj}\textrm{-}S)\ar@/_2pc/[l]_{F_{\Lambda}j_!Q_{S}F_{S}^{-1}}\ar@/^2pc/[l]^{F_{\Lambda}j_\ast F_{S}^{-1}}
\\
\ \mathbf{K}^{b}(\mathrm{Proj}\textrm{-}R)\ar@{^{(}->}[u]^{}\ar[r]^{\mathbf{D}^{b}(i_*)}&\ \mathbf{K}^{b}(\mathrm{Proj}\textrm{-}\Lambda)\ar@{^{(}->}[u]^{}\ar@/_2pc/[l]_{\mathbf{L}^{b}(i^*)}\ar@/^2pc/[l]^{\mathbf{D}^{b}(i^{!})}
\ar[r]^{\mathbf{D}^{b}(j^{\ast})} &\ \mathbf{K}^{b}(\mathrm{Proj}\textrm{-}S).\ar@{^(->}[u]^{}\ar@/_2pc/[l]_{\mathbf{L}^{b}(j_!)}\ar@/^2pc/[l]^{\mathbf{D}^{b}(j_\ast)}}\eqno{(5.2)}
$$
Here, the notation such as $Q_{\Lambda}$ means to take a special $dw \widetilde{\mathcal{P}}_{\Lambda}$-precover. The functors such as $F_{R}$ is the triangulate equivalence $\mathrm{Ho}(\mathcal{M}^{\mathrm{proj}}_{ctr,R})\rightarrow\mathbf{K}(\mathrm{Proj}\textrm{-}R)$, and $F^{-1}_{R}$ is the inverse of $F_{R}$. The notation such as $\mathbf{L}^{b}(i^*)$, $\mathbf{D}^{b}(i_*)$, $\mathbf{D}^{b}(i^{!})$, $\mathbf{L}^{b}(j_!)$, $\mathbf{D}^{b}(j_\ast)$, $\mathbf{D}^{b}(j^{\ast})$ are the derived functors of those in (5.1).
\end{cor}
\begin{proof}If $\mathrm{pd}_{R}M < \infty$ and $\mathrm{pd}M_{S} < \infty$, then by \cite{lu}, we have the recollement in the last row. By \cite[Proposition 6.5]{gj16b}, it suffices to construct the recollement (5.2). From \cite[Proposition 6.5]{gj16b}, we know that the triple
$\mathcal{M}^{\mathrm{proj}}_{ctr,R} = (dw \widetilde{\mathcal{P}}_{R},\mathcal{W}_{ctr,R},\mathbf{C}(R))$  and $\mathcal{M}^{\mathrm{proj}}_{ctr,S} = (dw \widetilde{\mathcal{P}}_{S},\mathcal{W}_{ctr,S},\mathbf{C}(S))$ are hereditary abelian
model structures on $\mathbf{C}(R)$ and $\mathbf{C}(S)$, respectively.
Note that the Hovey triple $(dw \widetilde{\mathcal{P}}_{R},\mathcal{W}_{ctr,R},\mathbf{C}(R))$ induces two complete cotorsion pairs $(dw \widetilde{\mathcal{P}}_{R},\mathcal{W}_{ctr,R})$ and
$(\widetilde{\mathcal{P}}_{R},\mathbf{C}(R))$. Similarly, we have  two complete cotorsion pairs $(dw \widetilde{\mathcal{P}}_{S},\mathcal{W}_{ctr,S})$ and
$(\widetilde{\mathcal{P}}_{S},\mathbf{C}(S))$ in $\mathbf{C}(S)$. Applying Theorem \ref{thm:1.1}, we obtain two cotorsion pairs
$(\mathcal{U}_{1},\mathcal{V}_{1})$ and $(\mathcal{U}_{2},\mathcal{V}_{2})$ in  $\mathbf{C}(\Lambda)$, where
\begin{align*}
\mathcal{U}_{1}&=\{\left(\begin{smallmatrix}  X^{\bullet}  \\   Y^{\bullet} \\\end{smallmatrix}\right)_{\phi}\in{\mathbf{C}(\Lambda)} \ \mid \ \mathrm{coker}\phi\in{dw \widetilde{\mathcal{P}}_{R}}, Y^{\bullet}\in{dw \widetilde{\mathcal{P}}_{S}}, \ M\otimes _{S}Y^{\bullet}  \to X^{\bullet} \ \textrm{is} \ \textrm{a} \ \textrm{monomorphism}\};\\
\mathcal{V}_{1}&=\{\left(\begin{smallmatrix}  X^{\bullet}  \\   Y^{\bullet} \\\end{smallmatrix}\right)_{\phi}\in{\mathbf{C}(\Lambda)} \ \mid \ X^{\bullet}\in{\mathcal{W}_{ctr,R}}, Y^{\bullet}\in{\mathcal{W}_{ctr,S}}\};\\
\mathcal{U}_{2}&=\{\left(\begin{smallmatrix}  X^{\bullet}  \\   Y^{\bullet} \\\end{smallmatrix}\right)_{\phi}\in{\mathbf{C}(\Lambda)} \ \mid \ \mathrm{coker}\phi\in{\widetilde{\mathcal{P}}_{R}},  Y^{\bullet}\in{\widetilde{\mathcal{P}}_{S}}, \ M\otimes _{S}Y^{\bullet}  \to X^{\bullet} \ \textrm{is} \ \textrm{a} \ \textrm{monomorphism}\};\\
\mathcal{V}_{2}&={\mathbf{C}(\Lambda)}.
\end{align*}
From \cite[Theorem 3.1]{ah00} we know that a $\Lambda$-module $X=\left(\begin{smallmatrix}  X  \\   Y \\\end{smallmatrix}\right)_{\phi}$ is projective if and only if $Y$ is projective in $S$-$\mathrm{Mod}$, $\mathrm{coker}\phi$ is projective in $R$-$\mathrm{Mod}$ and $\phi: M\otimes _{S}Y  \to X$ is monic. Therefore, one can show $\mathcal{U}_{1}=dw \widetilde{\mathcal{P}}_{\Lambda}$. Moreover, since $(\mathcal{U}_{1},\mathcal{V}_{1})$ and $(\mathcal{U}_{2},\mathcal{V}_{2})$ are  cotorsion pairs, we have $\mathcal{V}_{1}=\mathcal{W}_{ctr,\Lambda}$ and $\mathcal{U}_{2}=\widetilde{\mathcal{P}}_{\Lambda}$. Hence
$\mathcal{U}_{1}\cap\mathcal{V}_{1}=\mathcal{U}_{2}\cap\mathcal{V}_{2}=\widetilde{\mathcal{P}}_{\Lambda}$.
 Thus Theorem \ref{thm:main1} yields the desired recollement. Note that in each model structures above, all objects are fibrant. Therefore, we have ${L}(i^{\ast})=i^{\ast}Q_{\Lambda}$, ${R}(i_{\ast})=i_{\ast}$, ${R}(i^{!})=i^{!}$, ${L}(j_{!})=j_!Q_{S}$, ${R}(j^{\ast})=j^{\ast}$ and ${R}(j_{\ast})=j_{\ast}$. Here, the notation such as $Q_{\Lambda}$ means to take a special $dw \widetilde{\mathcal{P}}_{\Lambda}$-precover. One can check that all diagrams are commutative.
\end{proof}

It is shown in \cite[Proposition 2.2.1(1)]{be14} that the projective cotorsion pair $(\mathrm{Proj}\textrm{-}R,\mathrm{Mod}\text{-}R)$ in $\mathrm{Mod}\text{-}R$ can be lifted to a complete hereditary cotorsion pair $(ex \widetilde{\mathcal{P}}_{R},\mathcal{V}_{prj,R})$ in $\mathbf{C}(R)$. By \cite[Proposition 7.3]{gj16b}, there is a hereditary abelian model structure $\mathcal{M}^{\mathrm{proj}}_{stb,R} = (ex \widetilde{\mathcal{P}}_{R},\mathcal{V}_{prj,R},\mathbf{C}(R))$ on $\mathbf{C}(R)$, which is called as the \emph{exact Proj model structure} on $\mathbf{C}(R)$. Moreover, its homotopy category
$\mathrm{Ho}(\mathcal{M}^{\mathrm{proj}}_{stb,R})$, called as the \emph{projective stable derived category} over $R$, is equivalent to $\mathbf{K}_{ex}(\mathrm{Proj}\textrm{-}R)$, where $\mathbf{K}_{ex}(\mathrm{Proj}\textrm{-}R)$
is the chain homotopy category of all exact complexes of projective modules.
Thus we have the following result.

\begin{cor}\label{corollary:5.6} Let $\Lambda=\left(\begin{smallmatrix}  R & M \\  0 & S \\\end{smallmatrix}\right)$ be an upper triangular matrix ring. If $M$ has finite flat dimension as a right $S$-module, then we have the following recollements of triangulated categories:
$$
\xymatrix@C=50pt@R=65pt{ \mathrm{Ho}(\mathcal{M}^{\mathrm{proj}}_{stb,R})\ar[d]^{F_{R}}\ar[r]^{i_*}&\ \mathrm{Ho}(\mathcal{M}^{\mathrm{proj}}_{stb,\Lambda})\ar[d]^{F_{\Lambda}}\ar@/_2pc/[l]_{i^*Q_{\Lambda}}\ar@/^2pc/[l]^{i^{!}}
\ar[r]^{j^{\ast}} &\ \mathrm{Ho}(\mathcal{M}^{\mathrm{proj}}_{stb,S})\ar[d]^{F_{S}}\ar@/_2pc/[l]_{j_!Q_{S}}\ar@/^2pc/[l]^{j_\ast}  \\ \mathbf{K}_{ex}(\mathrm{Proj}\textrm{-}R)\ar[r]^{F_{\Lambda}i_*F_{R}^{-1}}&\ \mathbf{K}_{ex}(\mathrm{Proj}\textrm{-}\Lambda)\ar@/_2pc/[l]_{F_{R}i^*Q_{\Lambda}F_{\Lambda}^{-1}}\ar@/^2pc/[l]^{F_{R}i^{!}
F^{-1}_{\Lambda}}
\ar[r]^{F_{S}j^{\ast}F_{\Lambda}^{-1}} &\ \mathbf{K}_{ex}(\mathrm{Proj}\textrm{-}S)\ar@/_2pc/[l]_{F_{\Lambda}j_!Q_{S}F_{S}^{-1}}\ar@/^2pc/[l]^{F_{\Lambda}j_\ast F_{S}^{-1}}.}\eqno{(5.3)}$$
Here, the notation such as $Q_{\Lambda}$ means to take a special $ex \widetilde{\mathcal{P}}_{\Lambda}$-precover. The functors such as $F_{R}$ is the triangulate equivalence $\mathrm{Ho}(\mathcal{M}^{\mathrm{proj}}_{stb,R})\rightarrow\mathbf{K}_{ex}(\mathrm{Proj}\textrm{-}R)$, and $F^{-1}_{R}$ is the inverse of $F_{R}$.
\end{cor}

\begin{proof} It suffices to construct the recollement (5.3) by \cite[Proposition 6.5]{gj16b}.
From \cite[Proposition 6.5]{gj16b} we know that the triple
$\mathcal{M}^{\mathrm{proj}}_{stb,R} = (ex \widetilde{\mathcal{P}}_{R},\mathcal{V}_{prj,R},\mathbf{C}(R))$  and $\mathcal{M}^{\mathrm{proj}}_{stb,S} = (ex \widetilde{\mathcal{P}}_{S},\mathcal{V}_{prj,S},\mathbf{C}(S))$ are hereditary abelian
model structures on $\mathbf{C}(R)$ and $\mathbf{C}(S)$, respectively.
Note that the Hovey triple $(ex \widetilde{\mathcal{P}}_{R},\mathcal{V}_{prj,R},\mathbf{C}(R))$ induces two complete cotorsion pairs $(ex \widetilde{\mathcal{P}}_{R},\mathcal{V}_{prj,R})$ and
$(\widetilde{\mathcal{P}}_{R},\mathbf{C}(R))$. Similarly, we have  two complete cotorsion pairs $(ex \widetilde{\mathcal{P}}_{S},\mathcal{V}_{prj,S})$ and
$(\widetilde{\mathcal{P}}_{S},\mathbf{C}(S))$ in $\mathbf{C}(S)$. Applying Theorem \ref{thm:1.1}, we obtain two cotorsion pairs
$(\mathcal{U}_{1},\mathcal{V}_{1})$ and $(\mathcal{U}_{2},\mathcal{V}_{2})$ in  $\mathbf{C}(\Lambda)$, where
\begin{align*}
\mathcal{U}_{1}=&\{\left(\begin{smallmatrix}  X^{\bullet}  \\   Y^{\bullet} \\\end{smallmatrix}\right)_{\phi}\in{\mathbf{C}(\Lambda)} \ \mid \ \mathrm{coker}\phi\in{ex \widetilde{\mathcal{P}}_{R}}, Y^{\bullet}\in{ex \widetilde{\mathcal{P}}_{S}}, \ M\otimes _{S}Y^{\bullet}  \to X^{\bullet} \ \textrm{is} \ \textrm{a} \ \textrm{monomorphism}\};\\
\mathcal{V}_{1}=&\{\left(\begin{smallmatrix}  X^{\bullet}  \\   Y^{\bullet} \\\end{smallmatrix}\right)_{\phi}\in{\mathbf{C}(\Lambda)} \ \mid \ X^{\bullet}\in{\mathcal{V}_{prj,R}}, Y^{\bullet}\in{\mathcal{V}_{prj,S}}\};\\
\mathcal{U}_{2}=&\{\left(\begin{smallmatrix}  X^{\bullet}  \\   Y^{\bullet} \\\end{smallmatrix}\right)_{\phi}\in{\mathbf{C}(\Lambda)} \ \mid \ \mathrm{coker}\phi\in{\widetilde{\mathcal{P}}_{R}},  Y^{\bullet}\in{\widetilde{\mathcal{P}}_{S}}, \ M\otimes _{S}Y^{\bullet}  \to X^{\bullet} \ \textrm{is} \ \textrm{a} \ \textrm{monomorphism}\};\\
\mathcal{V}_{2}=&{\mathbf{C}(\Lambda)}.
\end{align*}
Consider the exact sequence of complexes $0\rightarrow M\otimes _{S}Y^{\bullet} \rightarrow X^{\bullet}\rightarrow \mathrm{coker}\phi\rightarrow 0$. Since $M$ has finite flat dimension as a right $S$-module by hypothesis, $X^{\bullet}$ is an exact complex. This means that
$\mathcal{U}_{1}=ex \widetilde{\mathcal{P}}_{\Lambda}$. Moreover, since $(\mathcal{U}_{1},\mathcal{V}_{1})$ and $(\mathcal{U}_{2},\mathcal{V}_{2})$ are  cotorsion pairs, we have $\mathcal{V}_{1}=\mathcal{V}_{prj,\Lambda}$ and $\mathcal{U}_{2}=\widetilde{\mathcal{P}}_{\Lambda}$. Hence
$\mathcal{U}_{1}\cap\mathcal{V}_{1}=\mathcal{U}_{2}\cap\mathcal{V}_{2}=\widetilde{\mathcal{P}}_{\Lambda}$.
 Thus Theorem \ref{thm:main1} yields the desired recollement. Note that in each model structures above, all objects are fibrant. Therefore, we have ${L}(i^{\ast})=i^{\ast}Q_{\Lambda}$, ${R}(i_{\ast})=i_{\ast}$, ${R}(i^{!})=i^{!}$, ${L}(j_{!})=j_!Q_{S}$, ${R}(j^{\ast})=j^{\ast}$ and ${R}(j_{\ast})=j_{\ast}$. Here, the notation such as $Q_{\Lambda}$ means to take a special $ex \widetilde{\mathcal{P}}_{\Lambda}$-precover. One can check that all diagrams are commutate.
\end{proof}
\subsection{The category of Gorenstein injective modules}
Let $\Lambda=\left(\begin{smallmatrix}  R & M \\  0 & S \\\end{smallmatrix}\right)$ be an upper triangular matrix ring, where $R$ and $S$ are rings and $_{R}M_{S}$ is an $R$-$S$-bimodule. Recall
that the \emph{monomorphism category} $\mathrm{Mon}(\Lambda)$ induced by the bimodule $_{R}M_{S}$ is the subcategory of Mod-$\Lambda$ consisting of $\left(\begin{smallmatrix}  X  \\  Y \\\end{smallmatrix}\right)_{\phi}$ such that $\phi:M\otimes_{S}Y\rightarrow X$ is a monomorphism in Mod-$R$.
For any  class $\mathcal{C}$ of $R$-modules and any class $\mathcal{D}$ of $S$-modules, let $\mathrm{Rep}(\mathcal{C},\mathcal{D})$ denote the following class of $\Lambda$-modules: $$\mathrm{Rep}(\mathcal{C},\mathcal{D})=\{N=\left(\begin{smallmatrix} X  \\  Y \\\end{smallmatrix}\right)_{\phi}\in\mathrm{Mod}\textrm{-}\Lambda\mid X \in\mathcal{C},~Y\in\mathcal{D} \}.$$

For any ring $R$, an $R$-module $X$ is \emph{Gorenstein injective} if there exists an exact
complex of injective $R$-modules $E^{\bullet} = \cdots E^{-1} \rightarrow E^{0} \rightarrow E^{1} \rightarrow\cdots$ such that for
any injective module $I$, the complex $\textrm{Hom}_{R}$$(I,E^{\bullet})$ is still exact, and such that
$X \cong {\rm ker}(E^{0} \rightarrow E^{1})$. Such an exact complex of injective modules $E^{\bullet}$ is called a \emph{totally acyclic complex of injective modules}. For any ring $R$, we denote by $\mathcal{GI}(R)$ (resp., $\mathcal{I}(R)$) the class of Gorenstein injective (resp., injective) $R$-modules and by $\mathcal{GI}(\mathrm{Mon}(\Lambda))$  the class of Gorenstein injective objects in $\mathrm{Mon}(\Lambda)$. Recently, \v{S}aroch and $\check{\rm{S}}$t'ov\'{\i}$\check{\rm{c}}$ek \cite{JJarxiv} have proved that $(^{\perp}\mathcal{GI}(R),\mathcal{GI}(R))$ is a complete hereditary cotorsion pair for any ring $R$.

The following result give a characterization of injective objects in $\mathrm{Mon}(\Lambda)$.

\begin{lem}\label{lem inj} Let $\Lambda=\left(\begin{smallmatrix}  R & M \\  0 & S \\\end{smallmatrix}\right)$ be an upper  triangular matrix ring. If $M_{S}$ is flat, then $X=\ccc{X_{1}}{X_{2}}{\phi}$ is an injective object in $\mathrm{Mon}(\Lambda)$ if and only if $X\in\mathrm{Rep}(\mathcal{I}(R),\mathcal{I}(S))$ and $\phi$ is monic.
\end{lem}
\begin{proof}The ``if" statement follows from the proof of \cite[Proposition 2.2(2)]{xz18}.
For the ``only if" statement, suppose that $X=\ccc{X_{1}}{X_{2}}{\phi}$ is an injective object in $\mathrm{Mon}(\Lambda)$.
 We have $\Ext_{R}^{1}(U,X_{1})\cong \Ext_{\Lambda}^{1}(\ccc{U}{0}{},\ccc{X_{1}}{X_{2}}{\phi})=0$ for every $R$-module $U$. Therefore, $X_{1}$ is an injective $R$-module. Since $M_{S}$ is flat,  by \cite[Lemma 3.2(5)]{zhu20}, we have $\Ext_{S}^{1}(V,X_{2})\cong \Ext_{\mathrm{Mon}(\Lambda)}^{1}(\ccc{M\oo_{S}V}{V}{},\ccc{X_{1}}{X_{2}}{\phi})=0$ for any $S$-module $V$. It follows that $X_{2}$ is an injective $S$-module.
\end{proof}
 A bimodule $_{R}M_{S}$  is called \emph{cocompatible} if the following three conditions hold:

$(C1)$ $M_{S}$ is flat;

$(C2)$ $M\otimes_{S}Y\in\mathcal{I}(R)$ for each injective right $S$-module $Y$;

$(C3)$ the functor $M\otimes_{S}-$ preserves totally acyclic complex of injective modules.

\begin{ex}
(1) If $R=S$ and $M=\underset{finite}{\underbrace{R \oplus\cdots\oplus R}}$, then $_{R}M_{R}$ is a cocompatible bimodule.

(2) For an ring $R$, let $S:=R \oplus\cdots\oplus R$, and $_{R}M_{S}:={_{R}S_{S}}$. Then $_{R}M_{S}$ is a cocompatible bimodule.

(3) For an ring $S$, let $R:=\underset{finite}{\underbrace{ S \oplus\cdots\oplus S}}$, and $_{R}M_{S}:={_{R}R_{S}}$. Then $_{R}M_{S}$ is a cocompatible bimodule.

(4) Suppose that $R$ and $S$ are quasi-Frobenius rings. If  $_{R}M_{S}$ is a bimodule with $_{R}M$ injective and $M_{S}$  finitely generated projective, then $_{R}M_{S}$ is a cocompatible bimodule.
In fact, any injective left $S$-injective module is an direct summand of $\mathrm{Hom}_{\mathbb{Z}}(S,\mathbb{Q/Z})^{I}$ for some set $I$. To show $(C2)$, it is only need  to show that $M\otimes_{S}\mathrm{Hom}_{\mathbb{Z}}(S,\mathbb{Q/Z})^{I}$ is injective.
Since $M_{S}$ is finitely presented, we have an isomorphism
$M\otimes_{S}\mathrm{Hom}_{\mathbb{Z}}(S,\mathbb{Q/Z})^{I}\cong\mathrm{Hom}_{\mathbb{Z}}(\mathrm{Hom}_{S}(M,S),\mathbb{Q/Z})^{I}.$
By the fact that $_{R}M$ is  injective  and $S_{S}$ is injective, we obtain that $\mathrm{Hom}_{S}(M,S)$ is a flat right $R$-module. Therefore, $\mathrm{Hom}_{\mathbb{Z}}(\mathrm{Hom}_{S}(M,S),\mathbb{Q/Z})^{I}$ is injective. $(C3)$ holds since each injective left $R$-module is projective.
\end{ex}
\begin{lem}\label{lem Ginj} Let $\Lambda=\left(\begin{smallmatrix}  R & M \\  0 & S \\\end{smallmatrix}\right)$ be an upper  triangular matrix ring and $X=\ccc{X_{1}}{X_{2}}{\phi^{X}}$ a $\Lambda$-module. If $_{R}M_{S}$ is cocompatible, then $X\in\mathcal{GI}(\mathrm{Mon}(\Lambda))$ if and only if $X\in\mathrm{Rep}(\mathcal{GI}(R),\mathcal{GI}(S))$ and $\phi^{X}$ is monic.
\end{lem}
\begin{proof}``$\Rightarrow$". If $X=\ccc{X_{1}}{X_{2}}{\phi^{X}}\in\mathcal{GI}(\mathrm{Mon}(\Lambda))$, then there exists an exact sequence of injective objects in $\mathrm{Mon}(\Lambda)$:
$~E^{\bullet}=\ccc{E^{\bullet}_{1}}{E^{\bullet}_{2}}{}:~\cdots \longrightarrow
\ccc{E^{-1}_{1}}{E^{-1}_{2}}{\phi^{-1}}\stackrel{\binom{\partial^{-1}_{1}}{\partial^{-1}_{2}}{}}
\longrightarrow\ccc{E^{0}_{1}}{E^{0}_{2}}{\phi^{0}}\stackrel{\binom{\partial^{0}_{1}}{\partial^{0}_{2}}}\longrightarrow
\ccc{E^{1}_{1}}{E^{1}_{2}}{\phi^{1}}\stackrel{}\longrightarrow
\cdots$
with $\ccc{ X_{1}}{ X_{2}}{\phi^{X}}\cong\mathrm{ker}\binom{\partial^{0}_{1}}{\partial^{0}_{2}}$.
By Lemma \ref{lem inj}, $E^{i}_{1}$ and $E^{i}_{2}$ are injective modules for each $i\in \mathbb{Z}$. So we have  an exact sequence of injective $R$-modules
$~E^{\bullet}_{1}$
such that ker$\partial^{0}_{1}\cong X_{1}$. We know from Lemma \ref{lem inj} that $\ccc{I}{0}{}$ is injective in $\mathrm{Mon}(\Lambda)$ for each injective $R$-module $I$, so the complex $\Hom_{R}(I,E^{\bullet}_{1})\cong\Hom_{\mathrm{Mon}(\Lambda)}(\ccc{I}{0}{},\ccc{E^{\bullet}_{1}}{E^{\bullet}_{2}}{})$
is exact. It follows that $X_{1}\in\mathcal{GI}(R)$.
Similarly, we have an exact sequence of injective $S$-modules
$~E^{\bullet}_{2}$
such that ker$\partial^{0}_{2}\cong X_{2}$.
Let $E$ be an injective $S$-module. Since $_{R}M_{S}$ is cocompatible, $M\otimes_{S}E$ is an injective $R$-module.  Applying Lemma \ref{lem inj} again, we get that $\ccc{M\otimes_{S}E}{E}{}$ is injective in $\mathrm{Mon}(\Lambda)$, so the complex $\Hom_{R}(E,E^{\bullet}_{2})\cong\Hom_{\mathrm{Mon}(\Lambda)}(\ccc{M\otimes_{S}E}{E}{},\ccc{E^{\bullet}_{1}}{E^{\bullet}_{2}}{})$
is exact. It follows that $X_{2}\in\mathcal{GI}(S)$.

``$\Leftarrow$". Consider the exact sequence
$0\rightarrow\ccc{M\oo_{S}X_{2}}{X_{2}}{}\rightarrow\ccc{X_{1}}{X_{2}}{}\rightarrow\ccc{\mathrm{coker}\phi^{X}}{0}{}\rightarrow0.
$
Since the class of Gorenstein injective objects is closed under extensions, it only need to show
that $\ccc{M\oo_{S}X_{2}}{X_{2}}{}$ and $\ccc{\mathrm{coker}\phi^{X}}{0}{}$ are belong to $\mathcal{GI}(\mathrm{Mon}(\Lambda))$.
 Let $\ccc{I_{1}}{I_{2}}{\phi^{I}}$ be an injective object in $\mathrm{Mon}(\Lambda)$. Since $_{R}M_{S}$ satisfies  $(C2)$, $M\otimes_{S}I_{2}$ is an injective $R$-module. So the exact sequence $0\rightarrow M\otimes_{S}I_{2}\rightarrow I_{1}\rightarrow \mathrm{coker}\phi^{I}\rightarrow0$ splits. It follows that $\mathrm{coker}\phi^{I}$ is an injective $R$-module and we have $\ccc{I_{1}}{I_{2}}{\phi^{I}}\cong \ccc{M\oo_{S}I_{2}}{I_{2}}{}\oplus\ccc{\mathrm{coker}\phi^{I}}{0}{}$.
Since $X_{2}$ is Gorenstein injective, there exists a totally acyclic complex of injective modules $E^{\bullet}_{2}$.
Therefore, by condition $(C3)$, the complex $\Hom_{\mathrm{Mon}(\Lambda)}(\ccc{\mathrm{coker}\phi^{I}}{0}{},\ccc{M\oo_{S}E^{\bullet}_{2}}{E^{\bullet}_{2}}{})
\cong\Hom_{R}(\mathrm{coker}\phi^{I},M\oo_{S}E^{\bullet}_{2})$ is exact.
On the other hand, by the fact that the functor  $Y\mapsto \ccc{M\oo_{S}Y}{Y}{Id}$ is fully faithful, the complex $\Hom_{\mathrm{Mon}(\Lambda)}(\ccc{M\oo_{S}I_{2}}{I_{2}}{},\ccc{M\oo_{S}E^{\bullet}_{2}}{E^{\bullet}_{2}}{})
\cong\Hom_{S}(I_{2},E^{\bullet}_{2})$ is exact. Together, we showed that for  each injective object $\ccc{I_{1}}{I_{2}}{\phi^{I}}$ in $\mathrm{Mon}(\Lambda)$, $\Hom_{\mathrm{Mon}(\Lambda)}(\ccc{I_{1}}{I_{2}}{\phi^{I}},\ccc{M\oo_{S}E^{\bullet}_{2}}{E^{\bullet}_{2}}{})
$ is exact. Now, we have $\ccc{M\oo_{S}X_{2}}{X_{2}}{}\cong\mathrm{ker}\ccc{\partial ^{0}_{M\oo_{S}E^{\bullet}_{2}}}{\partial^{0}_{E^{\bullet}_{2}}}{}$, which means  that $\ccc{M\oo_{S}X_{2}}{X_{2}}{}\in\mathcal{GI}(\mathrm{Mon}(\Lambda))$.

It suffices to show  $\ccc{\mathrm{coker}\phi^{X}}{0}{}\in\mathcal{GI}(\mathrm{Mon}(\Lambda))$. Consider the exact sequence $M\oo_{S}E^{\bullet}_{2}$, we have  $M\oo_{S}X_{2}\cong \mathrm{ker}\partial^{0}_{M\oo_{S}E^{\bullet}_{2}}$. So $M\oo_{S}X_{2}\in\mathcal{GI}(R)$.
 Since $\mathcal{GI}(R)$ is coresolving, $\mathrm{coker}\phi^{X}\in\mathcal{GI}(R)$. Then
 there exists a totally acyclic complex of injective modules $C^{\bullet}$ such that  $\mathrm{coker}\phi^{X}\cong \mathrm{ker}\partial^{0}_{C^{\bullet}}$. It follows that $\ccc{\mathrm{coker}\phi^{X}}{0}{}\cong\mathrm{ker}\ccc{\partial^{0}_{\mathbf{C}}}{0}{}$.
   Note that for each injective object $\ccc{I_{1}}{I_{2}}{\phi^{I}}$ in $\mathrm{Mon}(\Lambda)$, we have isomorphisms $\Hom_{\mathrm{Mon}(\Lambda)}(\ccc{M\oo_{S}I_{2}}{I_{2}}{},\ccc{C^{\bullet}}{0}{})=
   \Hom_{\mathrm{Mon}(\Lambda)}(j_{!}(I_{2}),\ccc{C^{\bullet}}{0}{})
\cong\Hom_{S}(I_{2},j^{\ast}(\ccc{C^{\bullet}}{0}{}))=0$ and
   $\Hom_{\mathrm{Mon}(\Lambda)}(\ccc{\mathrm{coker}\phi^{I}}{0}{},\ccc{C^{\bullet}}{0}{})
\cong\Hom_{R}(\mathrm{coker}\phi^{I},C^{\bullet})$.
   Therefore, the complex $\Hom_{\mathrm{Mon}(\Lambda)}(\ccc{I_{1}}{I_{2}}{\phi^{I}},\ccc{C^{\bullet}}{0}{})
 \cong\Hom_{\mathrm{Mon}(\Lambda)}(\ccc{M\oo_{S}I_{2}}{I_{2}}{}\oplus\ccc{\mathrm{coker}\phi^{I}}{0}{},\ccc{C^{\bullet}}{0}{})$ is exact. It follows that $\ccc{\mathrm{coker}\phi^{X}}{0}{}\in\mathcal{GI}(\mathrm{Mon}(\Lambda))$, as desired.
\end{proof}

\begin{cor}\label{corollary:5.8} Let $\Lambda=\left(\begin{smallmatrix}  R & M \\  0 & S \\\end{smallmatrix}\right)$ be an upper  triangular matrix ring.  Denote by $\mathcal{I}(\mathrm{Mon}(\Lambda))$  the class of injective objects in $\mathrm{Mon}(\Lambda)$.
If $_{R}M_{S}$ is cocompatible,  then we have the following recollement of triangulated categories:
$$\xymatrix@C=50pt@R=50pt{\ \displaystyle\frac{\mathcal{GI}(R)}{\mathcal{I}(R)}\ar[r]^{i_*}&\ \displaystyle\frac{\mathcal{GI}(\mathrm{Mon}(\Lambda))}{\mathcal{I}(\mathrm{Mon}(\Lambda))}\ar@/_2pc/[l]_{i^*}\ar@/^2pc/[l]^{i^{!}R_{\Lambda}}
\ar[r]^{j^{\ast}} &\ \displaystyle\frac{\mathcal{GI}(S)}{\mathcal{I}(S)}\ar@/_2pc/[l]_{j_!}\ar@/^2pc/[l]^{j_\ast R_{S}},}$$
Here, the notation such as $R_{S}$ means to take a special $\mathcal{GI}(S)$-preenvelope.

\end{cor}
\begin{proof}
Note that $^{\perp}\mathcal{GI}(R)\cap{\mathcal{GI}(R)}=\mathcal{I}(R)$ by \cite[Theorem 5.2]{gj16}. Therefore, by Lemma \ref{lem4.2}, there exists a category $\mathcal{W}(R)$ of modules, such that $(R\text{-}\mathrm{Mod},\mathcal{W}(R),\mathcal{GI}(R))$ forms a hereditary abelian model structure.
Let $\mathcal{M}_{R}=(\mathrm{Mod}\text{-}R,\mathcal{W}(R),\mathcal{GI}(R))$ and $\mathcal{M}_{S}=(\mathrm{Mod}\text{-}S,\mathcal{W}(S),\mathcal{GI}(S))$ be abelian model structures on $\mathrm{Mod}\text{-}R$ and $\mathrm{Mod}\text{-}S$, respectively.
 Since $M_{S}$ is flat, applying Theorem \ref{thm:1.1}, we obtain two cotorsion pairs
$(\mathcal{U}_{1},\mathcal{V}_{1})$ and $(\mathcal{U}_{2},\mathcal{V}_{2})$ in  $\mathrm{Mod}\text{-}\Lambda$, where
\begin{align*}
\mathcal{U}_{1}=&\{\left(\begin{smallmatrix}  X  \\   Y \\\end{smallmatrix}\right)_{\phi}\in{\mathrm{Mod}\text{-}\Lambda} \ \mid \ \phi  \ \textrm{is} \ \textrm{a} \ \textrm{monomorphism}\}=\mathrm{Mon}(\Lambda);\\
\mathcal{V}_{1}=&\{\left(\begin{smallmatrix}  X  \\   Y \\\end{smallmatrix}\right)_{\phi}\in{\mathrm{Mod}\text{-}\Lambda} \ \mid \ \ X\in{\mathcal{I}(R)}, Y\in{\mathcal{I}(S)}\}= \mathrm{Rep}(\mathcal{I}(R),\mathcal{I}(S));\\
\mathcal{U}_{2}=&\{\left(\begin{smallmatrix}  X  \\   Y \\\end{smallmatrix}\right)_{\phi}\in{\mathrm{Mod}\text{-}\Lambda} \ \mid \ \mathrm{coker}\phi\in{\mathcal{W}(R)}, Y\in{\mathcal{W}(S)}, \phi \ \textrm{is} \ \textrm{a} \ \textrm{monomorphism}\};\\
\mathcal{V}_{2}=&\{\left(\begin{smallmatrix}  X  \\   Y \\\end{smallmatrix}\right)_{\phi}\in{\mathrm{Mod}\text{-}\Lambda} \ \mid \ X\in{\mathcal{GI}(R)}, Y\in{\mathcal{GI}(S)}\}=\mathrm{Rep}(\mathcal{GI}(R),\mathcal{GI}(S)).
\end{align*}
In the following, we claim that $\mathcal{U}_{1}\cap\mathcal{V}_{1}=\mathcal{U}_{2}\cap\mathcal{V}_{2}=\mathcal{I}(\mathrm{Mon}(\Lambda))$. Let $\left(\begin{smallmatrix}  X  \\   Y \\\end{smallmatrix}\right)_{\phi}\in\mathcal{U}_{2}\cap\mathcal{V}_{2}$. We show that $\left(\begin{smallmatrix}  X  \\   Y \\\end{smallmatrix}\right)_{\phi}\in\mathcal{U}_{1}\cap\mathcal{V}_{1}$. Note that
$\mathcal{W}(R)\cap{\mathcal{GI}(R)}=\mathcal{I}(R)$ and $\mathcal{W}(S)\cap{\mathcal{GI}(S)}=\mathcal{I}(S)$. So it suffices to prove that $X\in \mathcal{W}(R)$. Consider the exact sequence
$0\rightarrow M\otimes_{S}Y\rightarrow X\rightarrow \mathrm{coker}\phi\rightarrow0$. Since $_{R}M_{S}$ is cocompatible, $M\otimes_{S}Y\in\mathcal{I}(R)\in\mathcal{W}(R)$. Hence $X\in\mathcal{W}(R)$ since $\mathcal{W}(R)$ is closed under extensions. Conversely, let $\left(\begin{smallmatrix}  X  \\   Y \\\end{smallmatrix}\right)_{\phi}\in\mathcal{U}_{1}\cap\mathcal{V}_{1}$. It remains to prove that $\mathrm{coker}\phi\in{\mathcal{W}(R)}$.
By the argument above, we see that in this case,  $M\otimes_{S}Y\in\mathcal{I}(R)\in\mathcal{W}(R)$. Therefore, the claim follows by the fact that $\mathcal{W}(R)$ is a thick subcategory. Note that the triangular matrix ring $\Lambda$ can induce a recollement situation between the module categories over the rings $R$, $\Lambda$ and $S$ (see \cite[Example 2.7]{Psaroudakis} for instance).
Thus Lemma \ref{lem inj}, Lemma \ref{lem Ginj} and Theorem \ref{thm:main1} yield the desired recollement. Note that in each model structures above, all objects are cofibrant. Therefore, we have ${L}(i^{\ast})=i^{\ast}$, ${L}(i_{\ast})=i_{\ast}$, ${R}(i^{!})=i^{!}R_{\Lambda}$, ${L}(j_{!})=j_!$, ${L}(j^{\ast})=j^{\ast}$ and ${R}(j_{\ast})=j_{\ast}R_{S}$. Here, the notation such as $R_{S}$ means to take a special $\mathcal{GI}(S)$-preenvelope.
\end{proof}

\renewcommand\refname{References}

%
\end{document}